\documentclass[a4paper,11pt]{article}
\usepackage{amsmath,amsthm,amssymb,enumitem}
\usepackage[bookmarks=false]{hyperref}

\numberwithin{equation}{section}
{\theoremstyle{definition}\newtheorem{definition}{Definition}[section]

\newtheorem{remark}[definition]{Remark}
\newtheorem{remarkletter}{Remark}
\newtheorem{definitionletter}[remarkletter]{Definition}
}
\newtheorem{proposition}[definition]{Proposition}
\newtheorem{lemma}[definition]{Lemma}
\newtheorem{theorem}[definition]{Theorem}
\newtheorem{theoremletter}[remarkletter]{Theorem}

\newcommand{\Inn}{\operatorname{Inn}}
\newcommand{\M}{\operatorname{M}}
\newcommand{\SL}{\operatorname{SL}}
\newcommand{\rL}{\mathord{\text{\rm L}}}

\newcommand{\id}{\mathord{\operatorname{id}}}
\newcommand{\si}{\sigma}
\newcommand{\Tr}{\operatorname{Tr}}

\newcommand{\Ker}{\operatorname{Ker}}

\newcommand{\lspan}{\operatorname{span}}
\newcommand{\ot}{\otimes}

\newcommand{\dis}{\displaystyle}

\newcommand{\Mtil}{\widetilde{M}}

\newcommand{\Om}{\Omega}

\newcommand{\otalg}{\otimes_{\text{\rm alg}}}

\newcommand{\SO}{\operatorname{SO}}
\newcommand{\SU}{\operatorname{SU}}
\newcommand{\Prob}{\operatorname{Prob}}
\newcommand{\op}{^\text{\rm op}}
\newcommand{\cb}{_\text{\rm cb}}

\newcommand{\etatil}{\widetilde{\eta}}

\newcommand{\Lambdah}{\widehat{\Lambda}}

\newcommand{\bim}[3]{\mathord{\raisebox{-0.4ex}[0ex][0ex]{\scriptsize $#1$}{#2}\hspace{-0.05ex}\raisebox{-0.4ex}[0ex][0ex]{\scriptsize $#3$}}}

\newcommand{\Z}{\mathbb{Z}}
\newcommand{\R}{\mathbb{R}}
\newcommand{\C}{\mathbb{C}}
\newcommand{\recht}{\rightarrow}
\newcommand{\al}{\alpha}
\newcommand{\cF}{\mathcal{F}}
\newcommand{\actson}{\curvearrowright}
\newcommand{\Stab}{\operatorname{Stab}}
\newcommand{\ovt}{\mathbin{\overline{\otimes}}}
\newcommand{\cM}{\mathcal{M}}
\newcommand{\cMtil}{\widetilde{\mathcal{M}}}
\newcommand{\cK}{\mathcal{K}}
\newcommand{\cH}{\mathcal{H}}
\newcommand{\ol}[1]{\overline{#1}}
\newcommand{\abs}[1]{|#1|}
\newcommand{\cA}{\mathcal{A}}
\newcommand{\cB}{\mathcal{B}}
\newcommand{\cU}{\mathcal{U}}
\newcommand{\cG}{\mathcal{G}}
\newcommand{\eps}{\varepsilon}
\newcommand{\cN}{\mathcal{N}}
\newcommand{\cS}{\mathcal{S}}
\newcommand{\dpr}{^{\prime\prime}}
\newcommand{\F}{\mathbb{F}}
\newcommand{\Ad}{\operatorname{Ad}}
\newcommand{\N}{\mathbb{N}}
\newcommand{\cZ}{\mathcal{Z}}
\newcommand{\cP}{\mathcal{P}}
\newcommand{\norm}[1]{\|#1\|}
\newcommand{\Aut}{\operatorname{Aut}}
\newcommand{\om}{\omega}
\newcommand{\T}{\mathbb{T}}
\newcommand{\cL}{\mathcal{L}}
\newcommand{\Out}{\operatorname{Out}}
\newcommand{\cV}{\mathcal{V}}
\newcommand{\thetatil}{\widetilde{\theta}}
\newcommand{\gammah}{\widehat{\gamma}}
\newcommand{\Perm}{\operatorname{Perm}}
\newcommand{\cW}{\mathcal{W}}
\newcommand{\cC}{\mathcal{C}}
\newcommand{\vphi}{\varphi}

\begin{document}

\begin{center}
{\boldmath\LARGE\bf W$^*$-superrigidity for group von Neumann algebras \vspace{0.5ex}\\ of left-right wreath products}

\bigskip


{by Mihaita Berbec\footnote{KU Leuven, Department of Mathematics, mihai.berbec@wis.kuleuven.be \\ Supported by Research
    Programme G.0639.11 of the Research Foundation -- Flanders (FWO)} and Stefaan Vaes\footnote{KU Leuven, Department of Mathematics, stefaan.vaes@wis.kuleuven.be \\ Partially supported by ERC Starting Grant VNALG-200749, Research Programme G.0639.11 of the Research Foundation -- Flanders (FWO) and KU Leuven BOF research grant OT/08/032.}}
\end{center}

\begin{abstract}\noindent
We prove that for many nonamenable groups $\Gamma$, including all hyperbolic groups and all nontrivial free products, the left-right wreath product group $\cG := (\Z/2\Z)^{(\Gamma)} \rtimes (\Gamma \times \Gamma)$ is W$^*$-superrigid. This means that the group von Neumann algebra $L \cG$ entirely remembers $\cG$. More precisely, if $L \cG$ is isomorphic with $L \Lambda$ for an arbitrary countable group $\Lambda$, then $\Lambda$ must be isomorphic with $\cG$.
\end{abstract}

\section{Introduction and statements of the main results}

Over the last years, Popa's deformation/rigidity theory lead to a lot of progress in the classification of \emph{group measure space II$_1$ factors} $L^\infty(X) \rtimes G$ associated with free, ergodic, probability measure preserving actions of countable groups (cf.\ the surveys in \cite{Po06a,Va10a,Io12a}). In comparison, our understanding of group von Neumann algebras $L G$ is much more limited. Connes' theorem of \cite{Co76} implies that all II$_1$ factors $L G$ coming from \emph{amenable} groups $G$ with infinite conjugacy classes (icc) are isomorphic. Although \emph{nonamenable} groups with nonisomorphic group II$_1$ factors were already discovered in \cite{MvN43,Sc63,McD69}, the general question on how $L G$ depends on $G$ remains largely unanswered, especially when $G$ is a ``classical group'' like $\SL(n,\Z)$ or a free group $\F_n$.

The first rigidity phenomena for group von Neumann algebras emerged in \cite{Co80a}, and in \cite{Co80b}, Connes asked whether icc property (T) groups $G$ and $\Lambda$ with isomorphic group von Neumann algebras, $L G \cong L \Lambda$, must necessarily be isomorphic groups. Although this rigidity conjecture remains wide open, deformation/rigidity theory has provided large classes $\cC$ of icc groups  such that two groups $G$ and $\Lambda$ in the class $\cC$ must be isomorphic whenever they have isomorphic group II$_1$ factors, see e.g.\ \cite{Po01,Po04,IPP05,PV06}. This is for instance the case for the class $\cC$ of all wreath product groups $\Z/2\Z \wr \Gamma$ with $\Gamma$ an icc property (T) group, see \cite{Po04}. Note however that both $G$ and $\Lambda$ are assumed to belong to the class $\cC$, so that it is not excluded that $L G \cong L H$ for a group $H$ that is nonisomorphic with $G$ and that lies outside the class $\cC$. Even more so, in the case where $G = \Z/2\Z \wr \Gamma$ and $\Gamma$ is torsion-free, a nonisomorphic $H \not\cong G$ with $LH \cong L G$ always exists by \cite[Theorem 1.2]{IPV10}.

Only in \cite{IPV10}, the first \emph{W$^*$-superrigidity} theorem for group von Neumann algebras was established: for a large class of \emph{generalized wreath product groups} $\cG = (\Z / 2\Z)^{(I)} \rtimes \Gamma$, it was shown that if $L \cG \cong L \Lambda$ for an \emph{arbitrary} group $\Lambda$, then $\Lambda$ must be isomorphic with $\cG$. Such a group $\cG$ is called W$^*$-superrigid (see Definition \ref{def.Wstar-superrigid} for the precise terminology). So $\cG$ is W$^*$-superrigid if the group von Neumann algebra $L \cG$ ``remembers'' $\cG$.

The class of groups covered by \cite{IPV10} contains all $(\Z / 2 \Z)^{(I)} \rtimes (\Gamma \wr \Z)$, where $\Gamma$ is an arbitrary nonamenable group and $I = (\Gamma \wr \Z)/\Z$. In this paper, we extend the results of \cite{IPV10} and prove W$^*$-superrigidity for the more natural \emph{left-right wreath products} $\cG = (\Z / 2\Z)^{(\Gamma)} \rtimes (\Gamma \times \Gamma)$, where the direct product $\Gamma \times \Gamma$ acts on $\Gamma$ by left-right multiplication, and where $\Gamma$ is either the free group $\F_n$ with $n \geq 2$, or any icc hyperbolic group, or any nontrivial free product $\Gamma_1 * \Gamma_2$. The precise statement is given in Theorem \ref{thm.main-intro} below.

We expect that for most nonamenable icc groups $\Gamma$, the left-right wreath product group $(\Z/2\Z)^{(\Gamma)}\rtimes (\Gamma \times \Gamma)$ is W$^*$-superrigid. As we explain in Remark \ref{rem.McDuff}, this is however not true for arbitrary nonamenable icc groups $\Gamma$.

To prove our W$^*$-superrigidity theorem, we follow the approach of \cite{IPV10}, by considering the comultiplication $\Delta : L \Lambda \recht L \Lambda \ovt L \Lambda$ that is induced by another group von Neumann algebra decomposition $L \cG = L \Lambda$ and carefully analyzing how $\Delta$ relates to the initial von Neumann algebra structure of $L \cG$. The following are the two major steps in the proof. We first use Popa's malleable deformation for Bernoulli actions (see \cite{Po03}) and his spectral gap rigidity (see \cite{Po06b}) to prove that the subalgebra $L(\Gamma \times \Gamma) \subset L \cG$ is invariant under $\Delta$, up to unitary conjugacy. We next use the recent results on normalizers of amenable subalgebras in crossed products by hyperbolic groups (see \cite{PV12}), and in crossed products by arbitrary free product groups (see \cite{Io12b}), to prove that also the subalgebra $L\bigl((\Z/2\Z)^{(\Gamma)}\bigr) \subset L \cG$ is invariant under $\Delta$, up to unitary conjugacy. Both steps together bring us to a point where the general results of \cite{IPV10} can be applied.

Contrary to the approach of \cite{IPV10}, our proof does not use the clustering techniques of \cite{Po04}, but uses the recent results of \cite{PV12,Io12b} instead. As a consequence, we can also prove W$^*$-superrigidity for a number of subgroups of generalized wreath product groups. In particular, we let $H$ be any nontrivial torsion-free abelian group and let $\Gamma$, as above, be either the free group $\F_n$ with $n \geq 2$, or any icc hyperbolic group, or any nonamenable free product $\Gamma_1 * \Gamma_2$. We define $\cH_0$ as the subgroup of $H^{(\Gamma)}$ consisting of those elements $x$ with $\sum_g x_g = 0$. Then we prove that $\cH_0 \rtimes (\Gamma \times \Gamma)$ is always W$^*$-superrigid (see Theorem \ref{thm.main-intro}).

\begin{definitionletter}\label{def.Wstar-superrigid}
A countable group $\cG$ is called W$^*$-superrigid if the following holds: if $\Lambda$ is any countable group and if $\pi : L \Lambda \to (L \cG)^r$ is a $*$-isomorphism for some $r > 0$, then $r = 1$ and there exist an isomorphism of groups $\delta : \Lambda \recht \cG$, a character $\om : \Lambda \recht \T$ and a unitary $w \in L \cG$ such that
$$\pi(v_s) = \om(s) \, w \, u_{\delta(s)} \, w^* \quad\text{for all}\;\; s \in \Lambda \; .$$
Here $(v_s)_{s \in \Lambda}$ and $(u_g)_{g \in \cG}$ denote the canonical generating unitaries of $L \Lambda$, resp.\ $L \cG$.
\end{definitionletter}

The following is our main result. The proof is given at the end of Section \ref{sec.proof-main-thm}, as a consequence of the more general Theorem \ref{thm.main}.

\begin{theoremletter}\label{thm.main-intro}
Assume that $\Gamma$ is one of the following groups:
\begin{itemize}
\item an icc hyperbolic group,
\item a finitely generated, icc, nonamenable, discrete subgroup of a connected noncompact rank one simple Lie group with finite center,
\item a free product $\Gamma_1 * \Gamma_2$ with $|\Gamma_1| \geq 2$ and $|\Gamma_2| \geq 3$.
\end{itemize}
All of the following generalized wreath product groups $\cG$ are W$^*$-superrigid in the sense of Definition \ref{def.Wstar-superrigid}~:
\begin{enumerate}
\item the group $(\Z/n\Z)^{(\Gamma)} \rtimes (\Gamma \times \Gamma)$ where $n \in \{2,3\}$,
\item the kernel of the homomorphism $H^{(\Gamma)} \rtimes (\Gamma \times \Gamma) \recht H : x g \mapsto \sum_{k \in \Gamma} x_k$, where $H$ is an arbitrary nontrivial torsion-free abelian group.
\end{enumerate}
\end{theoremletter}

\begin{remarkletter}\label{rem.funny-case}
Let $\Gamma$ be a group as in Theorem \ref{thm.main-intro}. Assume moreover that $\Gamma$ has no nontrivial characters. Let $H$ be an an arbitrary nontrivial torsion-free abelian group and denote by $\cG_0$ the kernel of the homomorphism $H^{(\Gamma)} \rtimes (\Gamma \times \Gamma) \recht H$ given in Theorem \ref{thm.main-intro}. At the end of section \ref{sec.proof-main-thm}, we prove that $\cG_0$ has no characters either. So the conclusion of Theorem \ref{thm.main-intro} becomes stronger: whenever $\Lambda$ is a countable group and $\pi : L \Lambda \recht (L \cG_0)^r$ is a $*$-isomorphism, we have $r = 1$ and there exist an isomorphism of groups $\delta : \Lambda \recht \cG_0$ and a unitary $w \in L (\cG_0)$ such that $\pi(v_s) = w \, u_{\delta(s)} \, w^*$ for all $s \in \Lambda$.
\end{remarkletter}

\section{Preliminaries}

\subsection{Popa's intertwining-by-bimodules}

We recall Popa's intertwining-by-bimodules theorem. In the formulation of the theorem, we also introduce the notations $P \prec Q$ and $P \prec^f Q$ that are used throughout this article.

\begin{theorem}[{\cite[Theorem 2.1 and Corollary 2.3]{Po03}}]\label{thm.intertwining}
Let $(M,\tau)$ be a tracial von Neumann algebra. Assume that $p,q \in M$ are projections and that $P \subset pMp$ and $Q \subset qMq$ are von Neumann subalgebras with $P$ being generated by a group of unitaries $\cG \subset \cU(P)$. Then the following three statements are equivalent.
\begin{itemize}
\item There exist a nonzero partial isometry $v \in \M_{1,n}(\C) \ot p M q$, a projection $q_0 \in \M_n(\C) \ot Q$ and a normal $*$-homomorphism $\theta : P \recht q_0(\M_n(\C) \ot Q)q_0$ such that $x v = v \theta(x)$ for all $x \in P$.
\item There is no sequence of unitaries $(w_n)$ in $\cG$ satisfying
$$\|E_Q(x^* w_n y)\|_2 \recht 0 \quad\text{for all}\;\; x,y \in pMq \; .$$
\item There exists a nonzero $P$-$Q$-subbimodule of $p L^2(M) q$ that has finite right $Q$-dimension.
\end{itemize}
We write $P \prec Q$ if these equivalent conditions hold. We write $P \prec^f Q$ if $P p_0 \prec Q$ for all nonzero projections $p_0 \in P' \cap pMp$. Sometimes we write $P \prec_M Q$ to stress the ambient von Neumann algebra $M$.
\end{theorem}

Note that when the von Neumann algebra $M$ has a nonseparable predual, then sequences have to be replaced by nets in the formulation of Theorem \ref{thm.intertwining}.

%

\begin{lemma}[{\cite[Section 2]{Va10b}}] \label{lem.full-embed}
Let $\Gamma$ be a countable group and $\Gamma \actson (B,\tau)$ a trace preserving action. Put $M = B \rtimes \Gamma$. Let $p \in M$ be a projection and $P \subset pMp$ a von Neumann subalgebra.
\begin{enumerate}[label=(\alph*)]
\item\label{full-one} Assume that $\Lambda < \Gamma$ is a subgroup. The set of projections $p_0 \in P' \cap pMp$ satisfying $P p_0 \prec^f B \rtimes \Lambda$ attains its maximum in a projection $p_1$ that belongs to the center of the normalizer of $P$ inside $pMp$. Moreover $P (p-p_1) \not\prec B \rtimes \Lambda$.
\item\label{full-two} Assume that $\Lambda_1,\Lambda_2 < \Gamma$ are subgroups with $\Lambda_2 \lhd \Gamma$ being normal. If $P \prec^f B \rtimes \Lambda_j$ for all $j \in \{1,2\}$, then $P \prec^f B \rtimes (\Lambda_1 \cap \Lambda_2)$.
\end{enumerate}
\end{lemma}
\begin{proof}
The first statement follows from \cite[Proposition 2.6 and Lemma 2.5]{Va10b}, while the second statement follows from \cite[Lemmas 2.7 and 2.5]{Va10b}.
\end{proof}

We also need the following lemma.

\begin{lemma}\label{lem.embed-with-normalizer}
Let $\Gamma$ be a countable group and $\Gamma \actson (B,\tau)$ a trace preserving action. Put $M = B \rtimes \Gamma$ and let $p \in M$ be a projection. Assume that $Q \subset pMp$ is a von Neumann subalgebra that is normalized by a group of unitaries $\cG \subset \cU(pMp)$. Let $\Lambda < \Gamma$ be a subgroup.

If $Q \prec^f B$ and $\cG\dpr \prec B \rtimes \Lambda$, then $(Q \cup \cG)\dpr \prec B \rtimes \Lambda$.
\end{lemma}
\begin{proof}
For every subset $\cF \subset \Gamma$, we denote by $P_\cF$ the orthogonal projection of $L^2(M)$ onto the closed linear span of $\{b u_g \mid b \in B, g \in \cF\}$.
We say that a subset $\cF \subset \Gamma$ is small relative to $\Lambda$ if $\cF$ is contained in a finite union of subsets of the form $g \Lambda h$ with $g,h \in \Gamma$.

Assume that $(Q \cup \cG)\dpr \not\prec B \rtimes \Lambda$. Since $\cU(Q) \cG$ is a group of unitaries generating $(Q \cup \cG)\dpr$, we get from \cite[Lemma 2.4]{Va10b} sequences of unitaries $a_n \in \cU(Q)$ and $w_n \in \cG$ such that $\|P_{\cF}(a_n w_n)\|_2 \recht 0$ for every subset $\cF \subset \Gamma$ that is small relative to $\Lambda$.

Since $\cG\dpr \prec B \rtimes \Lambda$, Theorem \ref{thm.intertwining} provides a nonzero partial isometry $v \in \M_{1,n}(\C) \ot pM$, a projection $q \in \M_n(\C) \ot (B \rtimes \Lambda)$ and a normal $*$-homomorphism $\theta : \cG\dpr \recht q(\M_n(\C) \ot (B \rtimes \Lambda))q$ such that $x v = v \theta(x)$ for all $x \in \cG\dpr$. Denote $p_1 := vv^*$ and fix $0 < \eps < \|p_1\|_2 / 3$. By the Kaplansky density theorem, we can take a finite subset $\cF_1 \subset \Gamma$ and an element $v_1$ in the linear span of $\{ b u_g \mid b \in \M_{1,n}(\C) \ot B, g \in \cF_1 \}$ such that $\|v_1\| \leq 1$ and $\|v - v_1\|_2 < \eps$.

Denote $\cF_2 := \cF_1 \Lambda \cF_1^{-1}$. Observe that $\cF_2$ is small relative to $\Lambda$. Write $x_n := v_1 \theta(w_n) v_1^*$. By construction, every $x_n$ lies in the image of $P_{\cF_2}$ and we have that $\|x_n\| \leq 1$, $\|w_n p_1 - x_n\|_2 < 2\eps$ for all $n$.

Since $Q \prec^f B$, we obtain from \cite[Lemma 2.5]{Va10b} a finite subset $\cF_3 \subset \Gamma$ such that \linebreak $\|a_n - P_{\cF_3}(a_n)\|_2 < \eps$ for all $n$. In combination with the previous paragraph, we get that $\|a_n w_n p_1 - P_{\cF_3}(a_n) x_n\|_2 < 3\eps$ for all $n$.
Denote $\cF_4 := \cF_3 \cF_2$ and observe that $\cF_4$ is still small relative to $\Lambda$. By construction, $P_{\cF_3}(a_n) x_n$ lies in the image of $P_{\cF_4}$ and we have thus shown that $\|a_n w_n p_1 - P_{\cF_4}(a_n w_n p_1)\|_2 < 3 \eps$ for all $n$.

Since $\|P_{\cF}(a_n w_n)\|_2 \recht 0$ for every subset $\cF \subset \Gamma$ that is small relative to $\Lambda$, it follows from \cite[Lemma 2.3]{Va10b} that $\|P_{\cF_4}(a_n w_n p_1)\|_2 \recht 0$. Hence $\limsup_n \|a_n w_n p_1 \|_2 \leq 3\eps$. Since $a_n$ and $w_n$ are unitaries, we arrive at the contradiction that $\|p_1\|_2 \leq 3\eps < \|p_1\|_2$.
\end{proof}

\subsection{Bimodules and weak containment}

Let $(M,\tau)$ be a tracial von Neumann algebra and $Q \subset M$ a von Neumann subalgebra. The \emph{basic construction} $\langle M,e_Q \rangle$ is defined as the von Neumann algebra acting on $L^2(M)$ generated by $M$ and the orthogonal projection $e_Q$ of $L^2(M)$ onto $L^2(Q)$. Recall that $\langle M,e_Q\rangle$ equals the commutant of the right $Q$-action on $L^2(M)$, i.e.\ $\langle M,e_Q\rangle=B(L^2(M))\cap (Q\op)'$.

Let $M, N$ be tracial von Neumann algebras. An $M$-$N$-\emph{bimodule} $\bim{M}{\cH}{N}$ is a Hilbert space $\cH$ equipped with two commuting normal unital $*$-homomorphisms $M\to B(\cH)$ and $N\op \to B(\cH)$. Any $M$-$N$-bimodule $\bim{M}{\cH}{N}$ gives rise to a $*$-homomorphism $\pi_\cH :M \otalg N\op \to B(\cH)$ given by $\pi_\cH(x\otimes y\op)\xi=x\xi y$, for all $x\in M$, $y\in N$ and $\xi\in\cH$.

If $(\rho, \cK)$ and $(\pi,\cH)$ are unitary representations of a countable group $\Gamma$, we say that $\rho$ is \emph{weakly contained} in $\pi$ if $\norm{\rho(a)} \leq\norm{\pi(a)}$ for all $a\in\C\Gamma$.

Similarly, if $\bim{M}{\cK}{N}$ and $\bim{M}{\cH}{N}$ are $M$-$N$-bimodules, we say that $\bim{M}{\cK}{N}$ is \emph{weakly contained} in $\bim{M}{\cH}{N}$ if
$\norm{\pi_\cK(x)}\leq\norm{\pi_\cH(x)}$ for all $x\in M \otalg N\op$.

\subsection{Relative amenability}

A tracial von Neumann algebra $(M,\tau)$ is called \emph{amenable} if there exists an $M$-central state on $B(L^2(M))$ whose restriction to $M$ equals $\tau$.
Also $M$ is amenable if and only if the trivial $M$-$M$-bimodule $\bim{M}{L^2(M)}{M}$ is weakly contained in the coarse $M$-$M$-bimodule $\bim{M}{(L^2(M)\otimes L^2(M))}{M}$.

\begin{definition}[{\cite[Section 2.2]{OP07}}]
Let $(M,\tau)$ be a tracial von Neumann algebra and let $P\subset pMp$ and $Q\subset M$ be von Neumann subalgebras. We say that $P$ is amenable relative to $Q$, if there exists a $P$-central positive functional on the von Neumann algebra $p\langle M,e_Q\rangle p$ whose restriction to $pMp$ equals $\tau$.

Similarly, if $\Gamma$ is a countable group with subgroups $\Lambda_1,\Lambda_2 < \Gamma$, we say that $\Lambda_1$ is amenable relative to $\Lambda_2$ if the action of $\Lambda_1$ on $\Gamma/\Lambda_2$ by left translations admits an invariant mean.
\end{definition}

The following lemma is essentially contained in \cite[Proposition 6]{MP03}. For completeness, we provide a full proof.

\begin{lemma}\label{lem.rel-amen-subgroups}
Let $\Gamma$ be a countable group and $\Gamma \actson (B,\tau)$ a trace preserving action. Put $M = B \rtimes \Gamma$ and let $\Lambda_1,\Lambda_2 < \Gamma$ be subgroups. Then the following statements are equivalent.
\begin{enumerate}[label=(\alph*)]
\item\label{rel-amen-a} $B \rtimes \Lambda_1$ is amenable relative to $B \rtimes \Lambda_2$ inside $M$.

\item\label{rel-amen-b} $L \Lambda_1$ is amenable relative to $B \rtimes \Lambda_2$ inside $M$.

\item\label{rel-amen-c} $\Lambda_1$ is amenable relative to $\Lambda_2$ inside $\Gamma$.
\end{enumerate}
\end{lemma}
\begin{proof}
\ref{rel-amen-a} $\Rightarrow$ \ref{rel-amen-b} is trivial.

\ref{rel-amen-b} $\Rightarrow$ \ref{rel-amen-c}. For every $g \in \Gamma$, we denote by $\delta_{g \Lambda_2} \in \ell^\infty(\Gamma/\Lambda_2)$ the function that is equal to $1$ in $g \Lambda_2$ and that is equal to $0$ elsewhere. There is a unique unital normal $*$-homomorphism
$$\pi : \ell^\infty(\Gamma/\Lambda_2) \recht \langle M, e_{B \rtimes \Lambda_2} \rangle \quad\text{satisfying}\quad \pi(\delta_{g \Lambda_2}) = u_g \, e_{B \rtimes \Lambda_2} \, u_g^* \quad\text{for all}\;\; g \in \Gamma \; .$$
By construction, $\pi$ conjugates the left translation action of $\Gamma$ on $\ell^\infty(\Gamma/\Lambda_2)$ with the action $(\Ad u_g)_{g \in \Gamma}$. Since $L \Lambda_1$ is amenable relative to $B \rtimes \Lambda_2$ inside $M$, we can take an $L \Lambda_1$-central state $\Om$ on $\langle M, e_{B \rtimes \Lambda_2} \rangle$. Then $\Om \circ \pi$ is a $\Lambda_1$-invariant state on $\ell^\infty(\Gamma/\Lambda_2)$. Hence \ref{rel-amen-c} holds.

\ref{rel-amen-c} $\Rightarrow$ \ref{rel-amen-a}. We denote by $\eta : \Gamma \recht \cU(\ell^2(\Gamma/\Lambda_2))$ the unitary representation of $\Gamma$ given by left translation operators.
We then turn the Hilbert space $L^2(M) \ot \ell^2(\Gamma/\Lambda_2)$ into an $M$-$M$-bimodule with the bimodule action given by
$$(b u_g) \cdot (x \ot \xi) \cdot y := b u_g x y \ot \eta_g \xi \quad\text{for all}\;\; b \in B, g \in \Gamma, x,y \in M, \xi \in \ell^2(\Gamma/\Lambda_2) \; .$$
Since \ref{rel-amen-c} holds, take a sequence of unit vectors $\xi_n \in \ell^2(\Gamma/\Lambda_2)$ satisfying $\lim_n \|\eta_g \xi_n - \xi_n\|_2 = 0$ for all $g \in \Lambda_1$. Then the sequence of vectors $1 \ot \xi_n \in L^2(M) \ot \ell^2(\Gamma/\Lambda_2)$ satisfies
\begin{align*}
& \langle x \cdot (1 \ot \xi_n) , 1 \ot \xi_n \rangle = \tau(x) \;\;\text{for all}\;\; x \in M \quad\text{and}\\
& \lim_n \| bu_g \cdot (1 \ot \xi_n) - (1 \ot \xi_n) \cdot bu_g\|_2 = 0 \;\;\text{for all}\;\; b \in \cU(B), g \in \Lambda_1 \; .
\end{align*}
Observe that there is a unique unitary operator
$$\theta : L^2(\langle M, e_{B \rtimes \Lambda_2} \rangle) \recht L^2(M) \ot \ell^2(\Gamma/\Lambda_2) \quad\text{satisfying}\quad \theta( b u_g \, e_{B \rtimes \Lambda_2} \, x) = b u_g x \ot \delta_{g \Lambda_2}$$
for all $b \in B, g \in \Gamma, x \in M$. This unitary $\theta$ is $M$-$M$-bimodular. Define $S_n \in L^2(\langle M, e_{B \rtimes \Lambda_2} \rangle)$ given by $S_n := \theta^{-1}(1 \ot \xi_n)$. Choose a state $\Om$ on $\langle M, e_{B \rtimes \Lambda_2} \rangle$ as a weak$^*$-limit point of the sequence of states $T \mapsto \langle T S_n,S_n \rangle$. By construction, $\Om(x) = \tau(x)$ for all $x \in M$ and $\Om$ is $\cG$-central, where $\cG = \{b u_g \mid b \in \cU(B), g \in \Lambda_1\}$. Using the Cauchy-Schwarz inequality, it follows that $\Om$ is $(B \rtimes \Lambda_1)$-central. So \ref{rel-amen-a} holds.
\end{proof}

We need two elementary lemmas.

\begin{lemma}\label{lem.rel-amen-max-proj}
Let $(M,\tau)$ be a tracial von Neumann algebra and let $P\subset pMp$ and $Q\subset M$ be von Neumann subalgebras. The set of projections $p_0 \in P' \cap pMp$ with the property that $P p_0$ is amenable relative to $Q$, attains its maximum in a projection $p_1$ that belongs to the center of the normalizer of $P$ inside $pMp$.
\end{lemma}
\begin{proof}
Denote by $\cP$ the set of projections $p_0 \in P' \cap pMp$ with the property that $P p_0$ is amenable relative to $Q$. If $p_0 \in \cP$ and $u \in  \cN_{pMp}(P)$, it is easy to check that $u p_0 u^* \in \cP$. It therefore suffices to prove the following two statements.

1.\ {\it If $p_0,p_1 \in \cP$, then $q := p_0 \vee p_1$ belongs to $\cP$.} For all $j \in \{0,1\}$, choose $P p_j$-central positive functionals $\Om_j$ on $p_j \langle M,e_Q \rangle p_j$ with the property that $\Om_j(x) = \tau(x)$ for all $x \in p_j M p_j$. Define the positive functional $\Om$ on $q \langle M,e_Q \rangle q$ by the formula $\Om(T) := \Om_0(p_0 T p_0) + \Om_1(p_1 T p_1)$. It is easy to check that $\Om$ is $Pq$-central and that the restriction of $\Om$ to $q M q$ is normal and faithful. By \cite[Theorem 2.1]{OP07}, we get that $Pq$ is amenable relative to $Q$.

2.\ {\it If $p_n$ is an increasing sequence in $\cP$ that converges strongly to $q$, then also $q \in \cP$.} Take $P p_n$-central positive functionals $\Om_n$ on $p_n \langle M,e_Q \rangle p_n$ with the property that $\Om_n(x) = \tau(x)$ for all $n \in \N$ and all $x \in p_n M p_n$. Choose a positive functional $\Om$ on $q \langle M,e_Q \rangle q$ as a weak$^*$ limit point of the sequence of functionals $T \mapsto \Om_n(p_n T p_n)$. By construction, $\Om$ is $P q$-central and $\Om(x) = \tau(x)$ for all $x \in q M q$. So $q \in \cP$.
\end{proof}

We also need the following special case of \cite[Proposition 2.7]{PV11}.

\begin{lemma}[{\cite[Proposition 2.7]{PV11}}] \label{lem.rel-amen-intersect}
Let $\Gamma$ be a countable group and $\Gamma \actson (B,\tau)$ a trace preserving action. Put $M = B \rtimes \Gamma$. Let $p \in M$ be a projection and $P \subset pMp$ a von Neumann subalgebra. Assume that $\Lambda_1,\Lambda_2 < \Gamma$ are subgroups with $\Lambda_2 \lhd \Gamma$ being normal. If $P$ is amenable relative to $B \rtimes \Lambda_j$ for all $j \in \{1,2\}$, then $P$ is amenable relative to $B \rtimes (\Lambda_1 \cap \Lambda_2)$.
\end{lemma}

We finally need the concept of a left amenable bimodule, see \cite[Theorem 2.2]{Si10} and \cite[Definition 2.3]{PV11}.

\begin{definition}
Let $(M,\tau)$ and $(N,\tau)$ be tracial von Neumann algebras. Let $P \subset M$ be a von Neumann subalgebra. An $M$-$N$-bimodule $\bim{M}{\cK}{N}$ is said to be left $P$-amenable if $B(\cK) \cap (N\op)'$ admits a $P$-central state whose restriction to $M$ equals $\tau$.
\end{definition}

If $(M,\tau)$ is a tracial von Neumann algebra and if $P \subset pMp$, $Q \subset M$ are von Neumann subalgebras, then by definition, $P$ is amenable relative to $Q$ if and only if the $pMp$-$Q$-bimodule $p L^2(M)$ is left $P$-amenable.

The following easy lemmas are essentially contained in \cite[Section 2.2]{OP07}. For completeness, we provide full proofs.

\begin{lemma}\label{lem.left-amen-part}
Let $(M,\tau)$ and $(N,\tau)$ be tracial von Neumann algebras. Let $P \subset M$ be a von Neumann subalgebra and $\bim{M}{\cK}{N}$ an $M$-$N$-bimodule. The following two statements are equivalent.
\begin{enumerate}[label=(\alph*)]
\item There exists a nonzero $P$-central positive functional on $B(\cK) \cap (N\op)'$ whose restriction to $M$ is normal.
\item There exists a nonzero projection $p \in P' \cap M$ such that the $pMp$-$N$-bimodule $\bim{pMp}{(p \cK)}{N}$ is left $Pp$-amenable.
\end{enumerate}
\end{lemma}
\begin{proof}
(a) $\Rightarrow$ (b).\ Let $\Om$ be a nonzero $P$-central positive functional on $\cN := B(\cK) \cap (N\op)'$ whose restriction to $M$, denoted by $\om$ is normal. Take $T \in L^1(M)^+$ such that $\om(x) = \tau(xT)$ for all $x \in M$. Note that $T \neq 0$. Since $\om$ is $P$-central, we have that $T \in L^1(P' \cap M)$. Take $\eps > 0$ small enough such that the spectral projection $p:=\chi_{(\eps,+\infty)}(T)$ is nonzero. Note that $p \in P' \cap M$ and that we can take $S \in p(P' \cap M)^+ p$ such that $T S = S T = p$. The formula $y \mapsto \Om(S^{1/2} y S^{1/2})$ defines $Pp$-central positive functional on $B(p \cK) \cap (N\op)'$ whose restriction to $pMp$ equals $\tau$. So $\bim{pMp}{(p \cK)}{N}$ is left $Pp$-amenable.

(b) $\Rightarrow$ (a).\ Assume that $p \in P' \cap M$ is a nonzero projection and that $\Om$ is a $Pp$-central positive functional on $B(p \cK) \cap (N\op)'$ whose restriction to $pMp$ equals $\tau$. Then the formula $y \mapsto \Om(p y p)$ defines a nonzero $P$-central positive functional on $B(\cK) \cap (N\op)'$ whose restriction to $M$ is normal.
\end{proof}

\begin{lemma}\label{lem.left-amen-sequence}
Let $(M,\tau)$ be a tracial von Neumann algebra with von Neumann subalgebra $P \subset M$. Let $\cK$ be an $M$-$M$-bimodule. Assume that $\xi_n \in \cK$ is a sequence of vectors and $\eps > 0$ such that
\begin{itemize}
\item $\|x \xi_n\| \leq \|x\|_2$ for all $x \in M$ and $n \in \N$,
\item $\|\xi_n\| \geq \eps$ for all $n \in \N$,
\item for all $x \in P$, we have that $\lim_n \| x \xi_n - \xi_n x\| = 0$.
\end{itemize}
Then there exists a nonzero projection $p \in P' \cap M$ such that the $pMp$-$M$-bimodule $p \cK$ is left $Pp$-amenable.
\end{lemma}
\begin{proof}
Choose a positive functional $\Om$ on $B(\cK) \cap (M\op)'$ as a weak$^*$ limit point of the sequence of positive functionals $y \mapsto \langle y \xi_n,\xi_n\rangle$. The conditions on $\xi_n$ imply that $\Om(x) \leq \tau(x)$ for all $x \in M^+$, that $\Om(1) \geq \eps^2$ and that $\Om$ is $P$-central. In particular, $\Om$ is nonzero and the restriction of $\Om$ to $M$ is normal. The conclusion now follows from Lemma \ref{lem.left-amen-part}.
\end{proof}

\begin{lemma}\label{lem.left-amen-direct-sum}
Let $(M,\tau)$ and $(N,\tau)$ be tracial von Neumann algebras. Let $P \subset M$ be a von Neumann subalgebra. Assume that for all $j \in \{1,\ldots,\ell\}$, we are given an $M$-$N$-bimodule $\cK_j$. If $\bigoplus_{j=1}^\ell \cK_j$ is a left $P$-amenable $M$-$N$-bimodule, then there exists a $j \in \{1,\ldots,\ell\}$ and a nonzero projection $p \in P' \cap M$ such that $p \cK_j$ is a left $Pp$-amenable $pMp$-$N$-bimodule.
\end{lemma}
\begin{proof}
Put $\cK := \bigoplus_{j=1}^\ell \cK_j$ and denote by $p_j$ the orthogonal projection of $\cK$ onto $\cK_j$. Let $\Om$ be a $P$-central state on $B(\cK) \cap (N\op)'$ whose restriction to $M$ equals $\tau$. Take $j \in \{1,\ldots,\ell\}$ such that $\Om(p_j) \neq 0$. Then the formula $y \mapsto \Om(p_j y p_j)$ defines a nonzero $P$-central positive functional on $B(\cK_j) \cap (N\op)'$ whose restriction to $M$ is smaller or equal than $\tau$ and hence normal. So the conclusion follows from Lemma \ref{lem.left-amen-part}.
\end{proof}

\subsection{Weak amenability and class $\cS$} \label{sec.class-S-wa}

We very briefly introduce weak amenability and bi-exactness (class $\cS$) for countable groups. We only use these concepts in the following way: the first two families of groups in Theorem \ref{thm.main-intro} are weakly amenable and in class $\cS$, so that we can apply the results of \cite{PV12} to them.

Recall from \cite{CH88} that a countable group $\Gamma$ is called \emph{weakly amenable} if $\Gamma$ admits a sequence of finitely supported functions $f_n : \Gamma \recht \C$ tending to $1$ pointwise and satisfying $\sup_n \|f_n\|\cb < \infty$. Here $\|f\|\cb$ is the Herz-Schur norm, i.e.\ the cb-norm of the linear map $L \Gamma \recht L \Gamma : u_g \mapsto f(g) u_g$.

Following \cite{Oz03} (see also \cite[Chapter 15]{BO08}), a group $\Gamma$ is said to be in \emph{class $\cS$} (or bi-exact) if $\Gamma$ is an exact group and if there exists a map $\mu : \Gamma \recht \Prob \Gamma$ from $\Gamma$ to the probability measures on $\Gamma$ satisfying
$$
\lim_{k \recht \infty} \|\mu(g k h) - g \cdot \mu(k) \|_1 = 0 \quad\text{for all}\;\; g,h \in \Gamma \; .
$$
It immediately follows that if $\Gamma$ belongs to class $\cS$ and if $\Lambda < \Gamma$ is an infinite subgroup, then the centralizer of $\Lambda$ inside $\Gamma$ is amenable. Ozawa's theorem in \cite{Oz03} says that much more is true: if $Q \subset L \Gamma$ is any diffuse von Neumann subalgebra, then the relative commutant $Q' \cap L \Gamma$ is amenable.

\subsection{Property Gamma, inner amenability and McDuff II$_1$ factors}\label{sec.inner-amen}

Recall that a II$_1$ factor $M$ is said to have \emph{property Gamma,} if $M$ admits a sequence of unitaries $u_n \in M$ such that $\tau(u_n) = 0$ for all $n$ and $\lim_n \|u_n x - x u_n \|_2 = 0$ for all $x \in M$.

Let $G$ be an icc group and denote $M := L G$. By \cite{Ef73}, if $M$ has property Gamma, then $G$ must be \emph{inner amenable,} meaning that the unitary representation $(\Ad g)_{g \in G}$ on $\ell^2(G - \{e\})$ has almost invariant vectors: there exists a sequence of unit vectors $\xi_n \in \ell^2(G - \{e\})$ such that $\lim_n \|(\Ad g)(\xi_n) - \xi_n\|_2 = 0$ for every $g \in G$. The converse can however fail, as was shown in \cite{Va09}.

Denote by $R$ the unique hyperfinite II$_1$ factor. A II$_1$ factor $M$ is said to be \emph{McDuff} if $M$ is isomorphic with $M \ovt R$. Every McDuff II$_1$ factor has property Gamma. By \cite{McD69}, a II$_1$ factor $M$ is McDuff if and only if $M$ admits two central sequences of unitaries $u_n,v_n \in M$ such that $\tau(u_n) = \tau(v_n) = \tau(u_n v_n u_n^* v_n^*) = 0$ for all $n$.

For every II$_1$ factor $M$, we denote by $\Aut(M)$ the group of automorphisms of $M$, which naturally is a Polish group. We denote by $\Inn(M) := \{\Ad u \mid u \in \cU(M)\}$ the normal subgroup of inner automorphisms and by $\Out(M) := \Aut(M) / \Inn(M)$ the quotient group. Then $M$ is non-Gamma if and only if $\Inn(M)$ is closed in $\Aut(M)$. In that case, $\Out(M)$ naturally becomes a Polish group as well.

\subsection{Weakly mixing actions and weakly mixing representations}

Recall that a unitary representation $\pi : \Gamma \recht \cU(H)$ is called \emph{weakly mixing} if $\pi$ has no nonzero finite-dimensional globally $(\pi(g))_{g \in \Gamma}$-invariant subspaces.

Similarly, a probability measure preserving (pmp) action $\Gamma \actson (X,\mu)$ is called weakly mixing if the associated unitary representation $\Gamma \actson L^2(X) \ominus \C 1$ is weakly mixing. If $\Gamma \actson (X,\mu)$ is a pmp action, then the following conditions are equivalent:
\begin{itemize}
\item $\Gamma \actson (X,\mu)$ is weakly mixing,
\item the diagonal action $\Gamma \actson X \times X : g \cdot (x,y) = (g \cdot x, g \cdot y)$ is ergodic,
\item whenever $\Gamma \actson (Y,\eta)$ is a pmp action and $F : X \times Y \recht \C$ is a measurable function that is invariant under the diagonal action $\Gamma \actson X \times Y : g \cdot (x,y) = (g \cdot x, g \cdot y)$, we have that $F$ is a.e.\ equal to a function that only depends on the $Y$-variable.
\end{itemize}

The following lemma is classical (see e.g.\ \cite[Proposition 2.3 and Lemma 2.4]{PV06} for a simple proof).

\begin{lemma}\label{lem.prel-weak-mixing}
Assume that the countable group $\Gamma$ acts on the countable set $I$. Let $(X_0,\mu_0)$ be an arbitrary nontrivial standard probability space. Then the following conditions are equivalent.
\begin{itemize}
\item For every $i \in I$, the orbit $\Gamma \cdot i$ is infinite.
\item For every finite subset $\cF \subset I$, there exists a $g \in \Gamma$ such that $g \cdot \cF \cap \cF = \emptyset$.
\item The unitary representation $\Gamma \actson \ell^2(I)$ is weakly mixing.
\item The generalized Bernoulli action $\Gamma \actson (X_0,\mu_0)^I$ is weakly mixing.
\end{itemize}
\end{lemma}

\section{Spectral gap rigidity for generalized Bernoulli actions}\label{sec.spectral-gap}

Let $G$ be a countable discrete group acting on a countable set $I$. Assume that $(A_0,\tau)$ is an arbitrary tracial von Neumann algebra. We denote by $A_0^I$ the tensor product, with respect to $\tau$, of copies of $A_0$ indexed by $I$. We let $G$ act on $A_0^I$ by the generalized Bernoulli action: denoting by $\pi_i : A_0 \recht A_0^I$ the embedding of $A_0$ as the $i$-th tensor factor, this generalized Bernoulli action $(\si_g)_{g \in G}$ is given by $\si_g \circ \pi_i = \pi_{g \cdot i}$ for all $g \in G$ and $i \in I$.
We consider the crossed product von Neumann algebra $M := A_0^I \rtimes G$. Whenever $\cF \subset I$, we write $\Stab \cF := \{g \in G \mid g \cdot i = i, \forall i \in \cF \}$.

In \cite{Po03,Po04}, Popa discovered his fundamental \emph{malleable deformation} for Bernoulli crossed products $M = A_0^G \rtimes G$ and used it to establish the first W$^*$-rigidity theorems in the case where $G$ has property (T). In \cite{Po06b}, Popa introduced his spectral gap methods to prove W$^*$-rigidity theorems for $A_0^G \rtimes G$ in the case where $G$ is a direct product of nonamenable groups. These methods and results have been generalized in many subsequent works (see e.g.\ \cite{PV06,Va07,Io10,IPV10}) and were in particular extended to cover certain \emph{generalized} Bernoulli actions, associated with general group actions $G \actson I$. So far, the spectral gap methods could only be employed under the assumption that $\Stab i$ is amenable for all $i \in I$ (see e.g.\ \cite[Corollary 4.3]{IPV10}). In this section, we show that it is actually sufficient to have a constant $\kappa > 0$ such that $\Stab \cF$ is amenable for all subsets $\cF \subset I$ with $|\cF| \geq \kappa$.

We use the following variant, due to \cite{Io06}, of Popa's malleable deformation for Bernoulli crossed products. Consider the free product $A_0 * L \Z$ with respect to the natural traces. Denote by $\Mtil := (A_0* L \Z)^I \rtimes G$ the corresponding generalized Bernoulli crossed product.

Define the self-adjoint $h \in L \Z$ with spectrum $[-\pi,\pi]$ such that $\exp(ih)$ equals the canonical generating unitary $u_1 \in L \Z$. Put $u_t := \exp(ith)$ and note that $u_t$ is a one-parameter group of unitaries with $|\tau(u_t)| < 1$ for all $t \neq 0$. As above we denote by $\pi_i : A_0 * L \Z \recht (A_0 * L \Z)^I$ the embedding as the $i$-th tensor factor. We can then define the malleable deformation $(\al_t)_{t \in \R}$ by automorphisms of $\Mtil$ given by $\al_t(u_g) = u_g$ and $\al_t(\pi_i(x)) = \pi_i(u_t x u_t^*)$ for all $g \in G$, $t \in \R$, $i \in I$ and $x \in A_0 * L \Z$.

Denote $\rho_t := |\tau(u_t)|^2$ and observe that $0 \leq \rho_t < 1$ for all $t \neq 0$. For every finite subset $\cF \subset I$, we denote by $\pi_\cF : A_0^\cF \recht A_0^I$ the natural embedding. Define the unital completely positive maps $\psi_t : M \recht M$ given by $\psi_t(x) = E_M(\al_t(x))$ for all $x \in M$. Whenever $a \in A_0^\cF$ is the elementary tensor given by $\dis a = \underset{i \in \cF}{\ot} a_i$ with $a_i \in A_0 \ominus \C 1$, we have
$$\psi_t(\pi_\cF(a) u_g) = \rho_{\mbox{}\hspace{-0.15ex}t}^{\mbox{}\,|F|} \; \pi_\cF(a) u_g \quad\text{for all}\;\; t \in \R, g \in G \; .$$
Therefore we consider the malleable deformation $(\al_t)_{t \in \R}$, and the corresponding completely positive maps $(\psi_t)_{t \in \R}$, as the \emph{tensor length deformation} of the generalized Bernoulli crossed product $M = A_0^I \rtimes G$.

\begin{theorem}\label{thm.spectral-gap-rigidity}
Let $G \actson I$ be an action of a countable group on a countable set. Assume that $\kappa,\ell > 0$ are integers and that $G_1,\ldots,G_\ell < G$ are subgroups with the following property: for every finite subset $\cF \subset I$ with $|\cF| \geq \kappa$, there exists an $i \in \{ 1,\ldots,\ell \}$ such that $\Stab \cF$ is amenable relative to $G_i$.

Assume that $(A_0,\tau)$ and $(N,\tau)$ are arbitrary tracial von Neumann algebras. Consider as above the generalized Bernoulli crossed product $M = A_0^I \rtimes G$ with its tensor length deformation $\al_t \in \Aut(\Mtil)$.

Assume that $p \in N \ovt M$ is a nonzero projection and that $P \subset p (N \ovt M) p$ is a von Neumann subalgebra such that for all nonzero projections $q \in P' \cap p (N \ovt M) p$ and all $i = 1,\ldots,\ell$, we have that $P q$ is nonamenable relative to $N \ovt (A_0^I \rtimes G_i)$.

Then
$$\sup_{b \in \cU(P' \cap p (N \ovt M) p)} \|(\id \ot \al_t)(b) - b \|_2 \quad\text{converges to $0$ as $t \recht 0$.}$$
\end{theorem}

Put $\cM := N \ovt M$ and $\cMtil := N \ovt \Mtil$.
The proof of Theorem \ref{thm.spectral-gap-rigidity} follows closely the proofs of \cite[Lemma 5.1]{Po06b} and \cite[Corollary 4.3]{IPV10}. The essential difference is that we replace the bimodule $\bim{\cM}{\rL^2(\cMtil \ominus \cM)}{\cM}$ by the following $\cM$-$\cM$-submodule
\begin{equation}\label{eq.Kkappa}
\cK^\kappa := \overline{\lspan} \left\{ \; x \ot \pi_\cF(a) u_g \; \middle| \; \parbox{8cm}{$x \in N$, $g \in G$, $\cF \subset I$ with $\kappa \leq |\cF| < \infty$,\vspace{0.4ex}\\ $a = \underset{i \in \cF}{\ot} a_i$ with $a_i \in A_0 * L \Z$ for all $i$ and with $a_i \in A_0 * L \Z \ominus A_0$ for at least $\kappa$ elements $i \in \cF$} \; \right\} \; .
\end{equation}

Before proving Theorem \ref{thm.spectral-gap-rigidity}, we need the following lemma.

\begin{lemma}\label{lem.our-weak-containment}
Under the assumptions of Theorem \ref{thm.spectral-gap-rigidity}, put $\cM_i := N \ovt (A_0^I \rtimes G_i)$. Then there exist $\cM_i$-$\cM$-bimodules $\cH_i$ such that the $\cM$-$\cM$-bimodule $\cK^\kappa$ is weakly contained in the $\cM$-$\cM$-bimodule $\bigoplus_{i=1}^\ell \bigl(L^2(\cM) \ot_{\cM_i} \cH_i\bigr)$.
\end{lemma}
\begin{proof}
Let $u\in L\Z$ be the canonical generating unitary. Let $\cA \subset A_0\ominus \C 1$ be an orthonormal basis of $L^2(A_0)\ominus\C 1$. Define
$\cB \subset A_0*L\Z$ given by
$$\cB:=\{u^{n_1}a_1u^{n_2}a_2\cdots u^{n_{k-1}}a_{k-1}u^{n_k}\mid k\geq 1,\;\;\text{and for all $j$},\;\; n_j\in\Z - \{0\},\; a_j\in \cA \} \; .$$
By construction, we have the following orthogonal decomposition of $L^2(A_0*L\Z)$ into $A_0$-$A_0$-subbimodules:
\begin{equation*}
L^2(A_0*L\Z)=L^2(A_0)\oplus \bigoplus_{b\in\cB}{\ol{A_0bA_0}} \; .
\end{equation*}
Fix $\cF\subset I$ finite, with $\abs{\cF}\geq \kappa$, and fix for all $i \in \cF$, $c_i\in \cB$. Denote
$$c:=1\otimes \pi_\cF\Bigl(\underset{i\in\cF}{\ot} c_i\Bigr)\in N \ovt (A_0*L\Z)^I \; .$$
Define the $\cM$-$\cM$-subbimodule of $\cK^\kappa$ given by $\cK^c := \ol{\cM c \cM}$. Define the subgroup $\Lambda < G$ given by
\begin{equation}\label{eq.form-Lambda}
\Lambda := \{g \in G \mid g \cdot \cF = \cF , c_{g \cdot i} = c_i \;\;\text{for all}\;\; i \in \cF \} \; .
\end{equation}
The formula $x \ot y \mapsto x c y$ defines an $\cM$-$\cM$-bimodular unitary between $\rL^2(\cM) \ot_Q \rL^2(\cM)$ and $\cK^c$ with $Q:= N \ovt (A_0^{I - \cF} \rtimes \Lambda)$. The different $\cK^c$ span a dense subspace of $\cK^\kappa$. Also, if $\cF,c$ and $\cF',c'$ are chosen as above, there are two possibilities: either there exists a $g \in G$ such that $\cF' = g \cdot \cF$ and $c'_{g \cdot i} = c_i$ for all $i \in \cF$, or such a $g \in G$ does not exist. In the first case, we have $\cK^c = \cK^{c'}$, while in the second case, we have $\cK^c \perp \cK^{c'}$.

Altogether we can choose a sequence of $c$'s as above, denoted $c_n$, such that $\cK^\kappa$ is the orthogonal direct sum of its subbimodules $\cK^{c_n}$. To each $c_n$ corresponds a finite subset $\cF_n \subset I$ satisfying $|\cF_n| \geq \kappa$, and a subgroup $\Lambda_n < G$ given by \eqref{eq.form-Lambda}. Note that by \eqref{eq.form-Lambda}, we get that $\Stab \cF_n$ is a finite index subgroup of $\Lambda_n$. Writing $Q_n = N \ovt (A_0^{I - \cF_n} \rtimes \Lambda_n)$, we conclude that $\cK^\kappa$ is isomorphic to the direct sum of the sequence of $\cM$-$\cM$-bimodules $L^2(\cM) \ot_{Q_n} L^2(\cM)$.

By the assumptions of the lemma, for every $n$, there exists an $i(n) \in \{1,\ldots,\ell\}$ such that $\Stab \cF_n$ is amenable relative to $G_{i(n)}$ inside $G$. Since $\Stab \cF_n < \Lambda_n$ has finite index, also $\Lambda_n$ is amenable relative to $G_{i(n)}$ inside $G$. It then follows from Lemma \ref{lem.rel-amen-subgroups} that $N \ovt (A_0^I \rtimes \Lambda_n)$ is amenable relative to $\cM_{i(n)}$. A fortiori, $Q_n$ is amenable relative to $\cM_{i(n)}$. By \cite[Proposition 2.4.3]{PV11}, this means that $\bim{\cM}{L^2(\cM)}{Q_n}$ is weakly contained in $\bim{\cM}{\bigl(L^2(\cM) \ot_{\cM_{i(n)}} L^2(\cM)\bigr)}{Q_n}$. Defining $\cH_i$ as the direct sum of all $L^2(\cM) \ot_{Q_n} L^2(\cM)$ with $i(n) = i$, it follows that $\cK^\kappa$ is weakly contained in $\bigoplus_{i=1}^\ell \bigl(L^2(\cM) \ot_{\cM_i} \cH_i\bigr)$ as an $\cM$-$\cM$-bimodule.
\end{proof}

\begin{proof}[{\bf Proof of Theorem \ref{thm.spectral-gap-rigidity}}]
Denote by $P_{\cK^\kappa}$ the orthogonal projection of $\rL^2(\cMtil)$ onto the closed subspace $\cK^\kappa$ that we defined in \eqref{eq.Kkappa}. Denote $\cU := \cU(P' \cap p (N \ovt M)p)$. We start by proving the following claim that is a variant of Popa's fundamental transversality property in \cite[Lemma 2.1]{Po06b}.

{\bf Claim.} If $\dis\sup_{b \in \cU} \|P_{\cK^\kappa}((\id \ot \al_t)(b))\|_2 \recht 0$ when $t \recht 0$, then also $\dis\sup_{b \in \cU} \|(\id \ot \al_t)(b) - b \|_2 \recht 0$ when $t \recht 0$.

To prove the claim, we first determine a formula for $\|P_{\cK^\kappa}(\id \ot \al_t)(y)\|_2$ when $y \in \cM$. For every $n \geq 0$, define the closed subspace $\cH_n \subset L^2(\cM)$ as
$$\cH_n := \overline{\lspan} \left\{ \; x \ot \pi_\cF(a) u_g \; \middle| \; \parbox{8cm}{$x \in N$, $g \in G$, $\cF \subset I$ finite, $|\cF| = n$, $a = \underset{i \in \cF}{\ot} a_i$ with $a_i \in A_0 \ominus \C 1$ for all $i \in \cF$} \; \right\} \; .$$
Observe that $L^2(\cM)$ is the orthogonal direct sum of the $\cH_n$. Denote by $P_n$ the orthogonal projection of $L^2(\cM)$ onto $\cH_n$.

Fix a finite subset $\cF \subset I$ with $|\cF| \geq \kappa$ and fix, for all $i \in \cF$, elements $a_i \in A_0 \ominus \C 1$. Put $a = \underset{i \in \cF}{\ot} a_i$. For all $x \in N$ and all $g \in G$, we have
\begin{align*}
x \ot \al_t(\pi_\cF(a) u_g) &= x \ot \pi_\cF\Bigl( \underset{i \in \cF}{\ot} u_t a_i u_t^*\Bigr) \; u_g \\
&= \sum_{\cG \subset \cF} \; x \ot \; \pi_\cG\Bigl(\underset{i \in \cG}{\ot} (u_t a_i u_t^* - \rho_t a_i)\Bigr) \; \pi_{\cF - \cG}\Bigl( \underset{i \in \cF - \cG}{\ot} \rho_t a_i \Bigr) \; u_g \; .
\end{align*}
In this last sum, the term corresponding to $\cG \subset \cF$ belongs to $\cK^\kappa$ if $|\cG| \geq \kappa$, and is orthogonal to $\cK^\kappa$ if $|\cG| < \kappa$. Therefore, we have for all $x \in N$ and all $g \in G$ that
$$
(1 - P_{\cK^\kappa})(x \ot \al_t(\pi_\cF(a) u_g)) \\ = \sum_{\parbox[t]{1cm}{\scriptsize $\cG \subset \cF$, \\ $|\cG| < \kappa$}} \; x \ot \; \pi_\cG\Bigl(\underset{i \in \cG}{\ot} (u_t a_i u_t^* - \rho_t a_i)\Bigr) \; \pi_{\cF - \cG}\Bigl( \underset{i \in \cF - \cG}{\ot} \rho_t a_i \Bigr) \; u_g \; .
$$
Put $y = x \ot \pi_\cF(a) u_g$ and assume that $y' = x' \ot \pi_{\cF'}(a') u_{g'}$ is of a similar form. Since
$$\langle u_t a u_t^* - \rho_t a , u_t b u_t^* - \rho_t b \rangle = (1-\rho_t^2) \, \tau(b^*a) \quad\text{for all}\;\; a,b \in A_0 \ominus \C 1 \; ,$$
we get that
$$\langle (1-P_{\cK^\kappa}) (\id \ot \al_t)(y) , (1-P_{\cK^\kappa})(\id \ot \al_t)(y') \rangle = \langle y, y' \rangle \; \sum_{j=0}^{\kappa-1} \bigl(\begin{smallmatrix} |\cF| \\ j \end{smallmatrix}\bigr) \, (1-\rho_t^2)^j \, \rho_t^{2(|\cF| - j)} \; ,$$
with both sides being zero if $\cF \neq \cF'$. We conclude that for all $y \in \cM$, we have
$$\|(1-P_{\cK^\kappa}) (\id \ot \al_t)(y)\|_2^2 = \sum_{n=0}^\infty c_\kappa(t,n) \, \|P_n(y)\|_2^2$$
where
$$c_\kappa(t,n) = \sum_{j=0}^{\min(\kappa - 1,n)} \bigl(\begin{smallmatrix} n \\ j \end{smallmatrix}\bigr) \, (1-\rho_t^2)^j \, \rho_t^{2(n-j)} \; .$$
Note that $c_\kappa(t,n) = 1$ if $n < \kappa$. It follows that
\begin{equation}\label{eq.ourformula}
\|P_{\cK^\kappa}(\al_t(y))\|_2^2 = \sum_{n=0}^\infty (1-c_\kappa(t,n)) \, \|P_n(y)\|_2^2 \quad\text{for all}\;\; y \in \cM \; .
\end{equation}

To prove the claim, assume that
$$\sup_{b \in \cU} \|P_{\cK^\kappa}(\id \ot \al_t)(b)\|_2 \recht 0 \quad\text{when}\quad t \recht 0 \; .$$
Choose $\eps > 0$. Take $t > 0$ such that $\|P_{\cK^\kappa}(\id \ot \al_t)(b)\|_2 < \eps$ for all $b \in \cU$. Since $c_\kappa(t,n) \recht 0$ when $n \recht \infty$ and $t$ is fixed, we can take $n_0$ such that $c_\kappa(t,n) < 1/2$ for all $n \geq n_0$. It then follows from \eqref{eq.ourformula} that for all $b \in \cU$, we have
\begin{equation}\label{eq.tussenstap}
\eps^2 > \|P_{\cK^\kappa}(\id \ot \al_t)(b)\|_2^2 \geq \frac{1}{2} \sum_{n=n_0}^\infty \|P_n(b)\|_2^2 \; .
\end{equation}
We finally take $s_0 > 0$ such that $1-\rho_s^n < \eps^2$ for all $|s| < s_0$ and all $0 \leq n < n_0$. Using \eqref{eq.tussenstap}, it follows that for all $b \in \cU$ and all $|s| < s_0$, we have
\begin{align*}
\|(\id \ot \al_s)(b) - b\|_2^2 & = \sum_{n=0}^\infty 2(1-\rho_s^n) \, \|P_n(b)\|_2^2 \\
& \leq \sum_{n=0}^{n_0-1} 2 \eps^2 \, \|P_n(b)\|_2^2 + 2\sum_{n=n_0}^\infty \|P_n(b)\|_2^2 \\
& \leq 2 \eps^2 + 4 \eps^2 \; .
\end{align*}
So, $\|(\id \ot \al_s)(b) - b \|_2 \leq 3\eps$ for all $|s| < s_0$ and all $b \in \cU$. This proves the claim.

To prove the theorem, assume that $\sup\{\|(\id \ot \al_t)(b) - b \|_2 \mid b \in \cU\}$ does not tend to $0$ as $t \recht 0$. We will produce a nonzero projection $q \in P' \cap p \cM p$ and a $j \in \{1,\ldots,\ell\}$ such that $P q$ is amenable relative to $\cM_j$. This will conclude the proof of the theorem.

By the claim above, we find an $\eps > 0$, a $t_0 > 0$, and for every $0 < t < t_0$, a unitary $b_t \in \cU$ such that $\|P_{\cK^\kappa}(\id \ot \al_t)(b_t)\|_2 \geq \eps$. Define $\xi_t := P_{\cK^\kappa}(\id \ot \al_t)(b_t)$. We have $\|\xi_t\|_2 \geq \eps$ for all $0 < t < t_0$. For every fixed $x \in P$, we have that $\|x \xi_t - \xi_t x\|_2 \recht 0$ as $t \recht 0$. We finally have $\|x \xi_t\|_2 \leq \|x\|_2$ for all $x \in \cM$. So Lemma \ref{lem.left-amen-sequence} provides a nonzero projection $q \in P' \cap p \cM p$ such that the $q \cM q$-$\cM$-bimodule $q \cK^\kappa$ is left $Pq$-amenable. Using \cite[Corollary 2.5]{PV11} and Lemma \ref{lem.our-weak-containment},
we find $\cM_j$-$\cM$-bimodules $\cH_j$ such that $\bigoplus_{j=1}^\ell q L^2(\cM) \ot_{\cM_j} \cH_j$ is left $Pq$-amenable. Making $q \in P' \cap p \cM p$ smaller, Lemma \ref{lem.left-amen-direct-sum} yields a $j \in \{1,\ldots,\ell\}$ such that $q L^2(\cM) \ot_{\cM_j} \cH_j$ is a left $Pq$-amenable bimodule. By \cite[Proposition 2.4.4]{PV11}, the $q\cM q$-$\cM_j$-bimodule $q L^2(\cM)$ is left $Pq$-amenable. This precisely means that $P q$ is amenable relative to $\cM_j$.
\end{proof}

We also need the following variant of \cite[Theorem 4.1]{Po03} and its subsequent generalizations in \cite[Theorem 2.1]{Io10} and \cite[Theorem 4.2]{IPV10}. Since our proof is almost identical, we are rather brief.

\begin{theorem}\label{thm.actual-conjugacy}
Let $G \actson I$ be an action of a countable group on a countable set. Assume that $(A_0,\tau)$ and $(N,\tau)$ are arbitrary tracial von Neumann algebra. Consider as above the generalized Bernoulli crossed product $M = A_0^I \rtimes G$ with its tensor length deformation $\al_t \in \Aut(\Mtil)$.

Assume that $p \in N \ovt M$ is a nonzero projection and that $Q \subset p (N \ovt M) p$ is a von Neumann subalgebra generated by a group of unitaries $\cG \subset \cU(Q)$ with the property that
$$\sup_{b \in \cG} \|(\id \ot \al_t)(b) - b \|_2 \quad\text{converges to $0$ as $t \recht 0$.}$$
If $G$ is icc, if $N$ is a factor and if for all $i \in I$, we have that $Q \not\prec N \ovt (A_0^I \rtimes \Stab i)$, then there exists a partial isometry $v \in N \ovt M$ with $vv^* = p$ and $v^* Q v \subset N \ovt L G$.
\end{theorem}
\begin{proof}
As above, we put $\cM = N \ovt M$ and $\cMtil = N \ovt \Mtil$.
We first prove the existence of a nonzero partial isometry $v \in \cM$ with the properties that $vv^* \in Q' \cap p \cM p$ and that $v^* Q v \subset N \ovt L G$. We reason exactly as in the proofs of \cite[Theorem 4.1]{Po03}, \cite[Theorem 2.1]{Io10} and \cite[Theorem 4.2]{IPV10}. For completeness, we nevertheless provide some details.

By the uniform convergence of $\id \ot \al_t$ on $\cG$, we find a $t > 0$ and a nonzero partial isometry $w_0 \in p \cMtil (\id \ot \al_t)(p)$ such that $x w_0 = w_0 (\id \ot \al_t)(x)$ for all $x \in Q$. We may assume that $t$ is of the form $t = 2^{-n}$. Since for all $i \in I$, we have that $Q \not\prec N \ovt (A_0^I \rtimes \Stab i)$, it follows from \cite[Lemma 4.1.1]{IPV10} that $w_0w_0^* \in \cM$ and $w_0^* w_0 \in (\id \ot \al_t)(\cM)$. Define the period two automorphism $\beta \in \Aut(\Mtil)$ given by $\beta(x) = x$ for all $x \in M$ and $\beta(\pi_i(u_1)) = u_1^*$ for all $i \in I$. By construction, $\beta \circ \al_t = \al_{-t} \circ \beta$.

We can now define
$$w_1 := (\id \ot \al_t)((\id \ot \beta)(v^*) v)$$
and check that $w_1$ is a nonzero partial isometry in $p \cMtil (\id \ot \al_{2t})(p)$ satisfying $x w_1 = w_1 (\id \ot \al_{2t})(x)$ for all $x \in Q$. Continuing inductively, we find a nonzero partial isometry $w \in p \cMtil (\id \ot \al_{1})(p)$ satisfying $x w = w (\id \ot \al_1)(x)$ for all $x \in Q$.
Literally repeating a part of the proof of \cite[Theorem 4.2]{IPV10}, we find a finite, possibly empty, subset $\cF \subset I$ such that $Q \prec N \ovt (A_0^\cF \rtimes \Stab \cF)$. Our assumption that $Q \not\prec N \ovt (A_0^I \rtimes \Stab i)$ for all $i \in I$, ensures that $\cF = \emptyset$. So, $Q \prec N \ovt L G$.

Take $n \in \N$, a nonzero partial isometry $v \in \M_{1,n}(\C) \ovt p(N \ovt M)$, a projection $q$ in \linebreak $\M_n(\C) \ovt N \ovt L G$ and a $*$-homomorphism $\theta : Q \recht q(\M_n(\C) \ovt N \ovt L G)q$ such that $x v = v \theta(x)$ for all $x \in Q$. Since $Q \not\prec N \ovt (A_0^I \rtimes \Stab i)$ for all $i \in I$, by \cite[Remark 3.8]{Va07}, we may assume that for all $i \in I$, we have $\theta(Q) \not\prec N \ovt L(\Stab i)$. By \cite[Lemma 4.1.1]{IPV10}, we then get that
$$\theta(Q)' \cap q(\M_n(\C) \ovt N \ovt M)q \subset \M_n(\C) \ovt N \ovt L G \; .$$
In particular, $v^* v$ is a projection in $\M_n(\C) \ovt N \ovt L G$ of trace at most $1$. Since $N \ovt L G$ is a II$_1$ factor, we may then assume that $n = 1$. So, we have found a nonzero partial isometry $v \in \cM$ with the properties that $vv^* \in Q' \cap p \cM p$ and that $v^* Q v \subset N \ovt L G$.

Let $v_n$ be a maximal sequence of nonzero partial isometries $v_n \in \cM$ with the property that the $v_n v_n^*$ are orthogonal projections in $Q' \cap p \cM p$ such that $v_n^* Q v_n \subset N \ovt L G$. Put $p_0 := p - \sum_n v_n v_n^*$. Since we can apply the previous paragraph to $Q p_0 \subset p_0 \cM p_0$, the maximality of the sequence $(v_n)$ ensures us that $p_0 = 0$.

Since $N \ovt L G$ is a II$_1$ factor and since the $v_n^* v_n$ form a sequence of projections in $N \ovt L G$ with $\sum_n v_n v_n^* = p$, we can take partial isometries $w_n \in N \ovt L G$ such that $w_n w_n^* = v_n^* v_n$ for all $n$ and such that the projections $w_n^* w_n$ are orthogonal. Then $v:=\sum_n v_n w_n$ is a partial isometry in $\cM$ with $vv^* = p$ and $v^* Q v \subset N \ovt L G$.
\end{proof}

\section{Properties of amplified comultiplications}\label{sec.comult}

Throughout this section, assume that $M_0$ is a II$_1$ factor and $r > 0$ such that $M_0^r = L\Lambda$ for some countable group $\Lambda$. We denote by $(v_s)_{s \in \Lambda}$ the canonical generating unitaries of $L \Lambda$ and define the \emph{comultiplication} $\Delta : L \Lambda \recht L \Lambda \ovt L \Lambda$ given by $\Delta(v_s) = v_s \ot v_s$ for all $s \in \Lambda$. Up to unitary conjugacy, we have a uniquely defined \emph{amplified comultiplication} $\Delta : M_0 \recht (M_0 \ovt M_0)^r$ that we continue to denote by $\Delta$.

At a certain point, we will need the explicit relation between the original comultiplication on $L \Lambda$ and the amplified comultiplication on $M_0$. This is spelt out in Remark \ref{rem.amplify-explicit}.

Whenever $(M,\tau)$ is a tracial von Neumann algebra and $M_0 \subset M$, we define as follows the inclusion $M_0^r \subset M^r$. Choose a projection $p \in \M_n(\C) \ot M_0$ with $(\Tr \ot \tau)(p) = r$ and define\linebreak $M_0^r := p (\M_n(\C) \ot M_0)p$ and $M^r := p (\M_n(\C) \ot M) p$. As such, the inclusion $M_0^r \subset M^r$ is defined up to conjugacy by a partial isometry in $\M_n(\C) \ot M_0$.

Apart from statement \ref{comult-three}, the following result is essentially contained in \cite[Proposition 7.2]{IPV10}. For completeness, we nevertheless give a full proof. At a first reading of Proposition \ref{prop.comult}, one may very well assume that $M_0 = M$, which is sufficient to prove Theorem \ref{thm.main-intro}.1. The most general setup is only needed to prove Theorem \ref{thm.main-intro}.2.

\begin{proposition}\label{prop.comult}
Let $M_0$ be a II$_1$ factor and $r > 0$ such that $M_0^r = L\Lambda$ for some countable group $\Lambda$. As above, denote by $\Delta : M_0 \recht (M_0 \ovt M_0)^r$ the amplified comultiplication. Assume that $M$ and $\Mtil$ are tracial von Neumann algebras such that $M_0 \subset M$ and $M_0 \subset \Mtil$.
\begin{enumerate}[label=(\alph*)]
\item\label{comult-two} If $P \subset M$ is a von Neumann subalgebra and $M_0 \not\prec_M P$, then $\Delta(M_0) \not\prec_{M \ovt M} M \ovt P$.
\item\label{comult-three} If $P \subset \Mtil$ is a von Neumann subalgebra and $\Delta(M_0)$ is amenable relative to $M^r \ovt P$ inside $M^r \ovt \Mtil$, then $M_0$ is amenable relative to $P$ inside $\Mtil$.
\item\label{comult-four} If $P \subset M_0$ is a von Neumann subalgebra that has no amenable direct summand, then for every nonzero projection $q \in \Delta(P)' \cap (M \ovt M)^r$, we have that $\Delta(P) q$ is nonamenable relative to $M^r \ot 1$.
\end{enumerate}
\end{proposition}
\begin{proof}
Throughout the proof, we fix a projection $p \in \M_n(\C) \ot M_0$ with $(\Tr \ot \tau)(p) = r$. We identify $p(\M_n(\C) \ot M_0)p = L \Lambda$.

\ref{comult-two}\ Let
$\Delta : L \Lambda \recht L \Lambda \ovt L \Lambda : \Delta(v_s) = v_s \ot v_s$
be the original comultiplication. Since $M_0 \not\prec_M P$, also $M_0^r \not\prec_M P$. By Theorem \ref{thm.intertwining}, we can take a sequence $s_n \in \Lambda$ such that
$$\|E_P(x^* v_{s_n} y)\|_2 \recht 0 \quad\text{for all}\;\; x,y \in p(\C^n \ot M) \; .$$
We claim that
\begin{equation}\label{eq.claim-zero}
\|E_{M \ovt P}(x^* \Delta(v_{s_n}) y)\|_2 \recht 0 \quad\text{for all}\;\; x,y \in p(\C^n \ot M) \ovt p(\C^n \ot M) \; .
\end{equation}
Indeed, \eqref{eq.claim-zero} is obvious when $x = x_1 \ot x_2$ and $y = y_1 \ot y_2$ are elementary tensors. Then \eqref{eq.claim-zero} follows easily for general $x,y$ as well. By \eqref{eq.claim-zero} and Theorem \ref{thm.intertwining}, we have $\Delta(L \Lambda) \not\prec_{M \ovt M} M \ovt P$. Then also the conclusion $\Delta(M_0) \not\prec_{M \ovt M} M \ovt P$ follows.

\ref{comult-three}\ We first state two preliminary observations.

($\ast$) Assume that $Q$ and $S$ are tracial von Neumann algebras and that $\bim{M^r}{\cH}{Q}$ and $\bim{{\Mtil}^r}{\cK}{S}$ are bimodules. If the $(M^r \ovt {\Mtil}^r)$-$(Q \ovt S)$-bimodule $\cH \ot \cK$ is left $\Delta(L \Lambda)$-amenable, then $\bim{{\Mtil}^r}{\cK}{S}$ is left $L \Lambda$-amenable.

To prove ($\ast$), assume that $\Om$ is a $\Delta(L \Lambda)$-central state on $B(\cH \ot \cK) \cap (Q\op \ovt S\op)'$ whose restriction to $M^r \ovt {\Mtil}^r$ equals the trace. Then the formula $\Om_0(T) := \Om(1 \ot T)$ defines a state on $B(\cK) \cap (S\op)'$ that is $(v_s)_{s \in \Lambda}$-central and whose restriction to ${\Mtil}^r$ equals the trace. In combination with the Cauchy-Schwarz  inequality, it follows that $\Om_0$ is actually $L \Lambda$-central. This concludes the proof of ($\ast$).

($\ast\ast$) Assume that $S$ is a tracial von Neumann algebra and that $\bim{\Mtil}{\cK}{S}$ is a bimodule. We leave it to the reader to check that $\bim{\Mtil}{\cK}{S}$ is left $M_0$-amenable if and only if the bimodule $\bim{{\Mtil}^r}{(p(\C^n \ot \cK))}{S}$ is left $M_0^r$-amenable.

We are now ready to prove \ref{comult-three}. By our assumptions, the bimodule $\bim{M^r \ovt \Mtil}{L^2(M^r \ovt \Mtil)}{M^r \ovt P}$ is left $\Delta(M_0)$-amenable. From ($\ast\ast$), we get that
$$\bim{M^r \ovt {\Mtil}^r}{L^2(M^r \ovt p(\C^n \ot \Mtil))}{M^r \ovt P}$$
is left $\Delta(M_0^r)$-amenable. It then follows from ($\ast$) that $\bim{{\Mtil}^r}{(p(\C^n \ot L^2(\Mtil)))}{P}$ is left $M_0^r$-amenable. Again using ($\ast\ast$), we get that $\bim{\Mtil}{\rL^2(\Mtil)}{P}$ is left $M_0$-amenable, i.e.\ that $M_0$ is amenable relative to $P$ inside $\Mtil$.

\ref{comult-four}\ Assume that $\bim{L \Lambda}{\cK}{L \Lambda}$ is an arbitrary bimodule. Denote by $\lambda : L (\Lambda) \recht B(\cK)$ and $\rho : (L \Lambda)\op \recht B(\cK)$ the normal $*$-homomorphisms given by the left, resp.\ right bimodule action. It is easy to check that there is a unique normal $*$-homomorphism
$$\Psi : L \Lambda \ovt (L \Lambda)\op \recht B(\cK \ot \cK \ot \cK) : \Psi(v_s \ot v_t\op) = \lambda(v_s)\rho(v_t\op) \ot \lambda(v_s) \ot \rho(v_t\op) \quad\text{for all}\;\; s,t \in \Lambda \; .$$
It follows in particular that the $L \Lambda$-$L \Lambda$-bimodule $\cK \ot \cK \ot \cK$ given by
$$v_s \cdot (\xi_1 \ot \xi_2 \ot \xi_3) \cdot v_t = (v_s \xi_1 v_t) \ot (v_s \xi_2) \ot (\xi_3 v_t)$$
is contained in a multiple of the coarse $L\Lambda$-$L\Lambda$-bimodule. Applying this statement to the bimodule $\bim{L \Lambda}{L^2(M^r)}{L \Lambda}$, it follows that the $\Delta(L \Lambda)$-$\Delta(L\Lambda)$-bimodule
$$\bim{\Delta(L\Lambda)}{\bigl(L^2(M^r \ovt M^r) \ot_{M^r \ot 1} L^2(M^r \ot M^r)\bigr)}{\Delta(L\Lambda)}$$
is contained in a multiple of the coarse $\Delta(L\Lambda)$-$\Delta(L\Lambda)$-bimodule. Then also
\begin{equation}\label{eq.this-is-coarse}
\bim{\Delta(M_0)}{\bigl(L^2(M^r \ovt M) \ot_{M^r \ot 1} L^2(M^r \ot M)\bigr)}{\Delta(M_0)}
\end{equation}
is contained in a multiple of the coarse $\Delta(M_0)$-$\Delta(M_0)$-bimodule.

Assume now that $q \in \Delta(P)' \cap (M \ovt M)^r$ is a nonzero projection such that $\Delta(P) q$ is amenable relative to $M^r \ot 1$. We must prove that $P$ has an amenable direct summand. By our assumption and \cite[Proposition 2.4.3]{PV11}, the bimodule $\bim{M^r \ovt M}{(L^2(M^r \ovt M)q)}{\Delta(P)q}$ is weakly contained in the bimodule
$$\bim{M^r \ovt M}{\bigl(L^2(M^r \ovt M) \ot_{M^r \ot 1} L^2(M^r \ovt M)q\bigr)}{\Delta(P)q} \; .$$
Viewing $L^2(\Delta(P)q)$ as a subspace of $L^2(M^r \ovt M)q$, it follows that $\bim{\Delta(P)q}{L^2(\Delta(P)q)}{\Delta(P)q}$ is weakly contained in the bimodule
$$\bim{\Delta(P)q}{\bigl(q L^2(M^r \ovt M) \ot_{M^r \ot 1} L^2(M^r \ot M)q\bigr)}{\Delta(P)q} \; .$$
Since the bimodule in \eqref{eq.this-is-coarse} is contained in a multiple of the coarse $\Delta(M_0)$-$\Delta(M_0)$-bimodule, we conclude that the trivial $\Delta(P)q$-$\Delta(P)q$-bimodule is weakly contained in the coarse\linebreak $\Delta(P)q$-$\Delta(P)q$-bimodule. Hence $\Delta(P)q$ has an amenable direct summand. Then also $P$ has an amenable direct summand.
\end{proof}

\begin{remark}\label{rem.amplify-explicit}
Assume that $M_0$ is a II$_1$ factor and $r > 0$ such that $M_0^r = L \Lambda$ for some countable group $\Lambda$. Consider the comultiplication
$$\Delta : L \Lambda \recht L \Lambda \ovt L \Lambda : \Delta(v_s) = v_s \ot v_s \quad\text{for all}\;\; s \in \Lambda \; .$$

Take a projection $p \in \M_n(\C) \ot M_0$ with $(\Tr \ot \tau)(p) = r$ and realize $M_0^r = p(\M_n(\C) \ot M_0)p$. Realize $(M_0 \ovt M_0)^r$ as $M_0^r \ovt M_0$. The relation between $\Delta$ and the amplified comultiplication $\Delta_0 : M_0 \recht M_0^r \ovt M_0$ can be concretized in the following slightly painful way.

Denote by $\zeta : \M_n(\C) \ot M_0 \recht M_0 \ot \M_n(\C)$ the flip isomorphism. Put
$$\Delta_1 := (\id \ot \id \ot \zeta^{-1}) \circ (\Delta_0 \ot \id) \circ \zeta \; ,$$
which is a unital $*$-homomorphism from $\M_n(\C) \ot M_0$ to $M_0^r \ovt \M_n(\C) \ovt M_0$. We then find an element $Z \in M_0^r \ovt \M_n(\C) \ovt M_0$ such that $Z^* Z = \Delta_1(p)$, $Z Z^* = p \ot p$ and $\Delta(x) = Z \Delta_1(x) Z^*$ for all $x \in M_0^r$.

\end{remark}

\section{Normalizers of relatively amenable subalgebras}\label{sec.norm-amen}

Throughout this section, we work in the following setup and under the following assumptions. We refer to Sections \ref{sec.class-S-wa} and \ref{sec.inner-amen} for the definitions of weak amenability, class $\cS$ and property Gamma.

{\bf Setup.} We are given a II$_1$ factor $M_0$, a countable group $\Lambda$ and a number $r > 0$ such that $M_0^r = L \Lambda$. We assume that $M_0 \subset M$, where $M$ is of the form $M = B \rtimes \Gamma$ for a given trace preserving action $\Gamma \actson (B,\tau)$ of a countable group $\Gamma$. We denote by $\Delta : M_0 \recht M_0^r \ovt M_0$ the amplified comultiplication, as in Section \ref{sec.comult}.

{\bf Assumptions.}
\begin{enumerate}
\item\label{assum.group} The group $\Gamma$ satisfies one of the following conditions.
\begin{enumerate}[label=(\alph*)]
\item\label{assum.group-one} $\Gamma$ is nonamenable, weakly amenable and in class $\cS$.
\item\label{assum.group-two} $\Gamma = \Gamma_1 * \Gamma_2$ with $|\Gamma_1| \geq 2$ and $|\Gamma_2| \geq 3$ and $M_0' \cap M^\omega = \C 1$.
\end{enumerate}

\item\label{assum.rel-com} We have $\Delta(M_0)' \cap M^r \ovt M = \C 1$.

\item\label{assum.embed} If $\Gamma_0 < \Gamma$ is a subgroup of infinite index, we have that $M_0 \not\prec_M B \rtimes \Gamma_0$.

\item\label{assum.rel-amen} We have that $M_0$ is nonamenable relative to $B$ inside $M$.
\end{enumerate}

At a first reading, one may very well assume that $M_0 = M$. In that case, assumption \ref{assum.rel-com} follows because $\Lambda$ is an icc group, while assumptions \ref{assum.embed} and \ref{assum.rel-amen} are trivially satisfied. This will be enough to prove Theorem \ref{thm.main-intro}.1. The general situation is only needed to prove Theorem \ref{thm.main-intro}.2.

The following theorem is a direct consequence of the main results in \cite{PV12} and \cite{Io12b}.

\begin{theorem}\label{thm.norm-amen}
Assume that we are in the setup and under the assumptions described above. If $P \subset M^r \ovt M$ is a von Neumann subalgebra such that $\Delta(M_0) \subset \cN_{M^r \ovt M}(P)\dpr$ and such that $P$ is amenable relative to $M^r \ovt B$, then $P \prec^f M^r \ovt B$.
\end{theorem}
\begin{proof}
Throughout the proof, we view $M^r \ovt M$ as the crossed product $(M^r \ovt B) \rtimes \Gamma$. By assumption \ref{assum.rel-com}, we have that $\Delta(M_0)' \cap (M \ovt M)^r = \C 1$. Since $\Delta(M_0) \subset \cN_{M^r \ovt M}(P)\dpr$, by Lemma \ref{lem.full-embed}.\ref{full-one}, it suffices to prove that $P \prec M^r \ovt B$.

First assume that $\Gamma$ satisfies assumption \ref{assum.group}.\ref{assum.group-one}. By \cite[Theorem 1.4]{PV12}, we have that either $\Delta(M_0)$ is amenable relative to $M^r \ovt B$, or that $P \prec M^r \ovt B$. Using Proposition \ref{prop.comult}.\ref{comult-three} and assumption \ref{assum.rel-amen}, we see that the first option is impossible. So we indeed get that $P \prec M^r \ovt B$.

Next assume $\Gamma$ satisfies assumption \ref{assum.group}.\ref{assum.group-two}. We apply the main results of \cite{Io12b} and need to introduce some of the corresponding notations. We extend the action $\Gamma \actson B$ to an action $\Gamma * \F_2 \actson B$ by letting $\F_2$ act trivially. Denote $\Mtil := B \rtimes (\Gamma * \F_2)$. View $\F_2$ as the free product of two copies of $\Z$ that we denote as $(\Z)_j$, $j=1,2$.
Choose self-adjoint elements $h_j$ in $L (\Z)_j$ with spectrum $[-\pi,\pi]$ and with the property that $\exp(i h_j)$ is the canonical unitary that generates $L (\Z)_j$. Denote $u^j_t := \exp(i t h_j)$ and define the one parameter group of automorphisms $\theta_t$ of $\Mtil$ given by $\theta_t(x) = u^j_t x (u^j_t)^*$ for all $x \in B \rtimes (\Gamma_j * (\Z)_j)$. Denote by $\cL$ the kernel of the natural surjective homomorphism of $\Gamma * \F_2$ onto $\F_2$ and put $N := B \rtimes \cL$. Observe that we can view $M^r \ovt \Mtil$ as the crossed product $M^r \ovt \Mtil = (M^r \ovt N) \rtimes \F_2$.

Fix $t \in (0,1)$. Since $P$ is amenable relative to $M^r \ovt B$ inside $M^r \ovt M$, we have that $(\id \ot \theta_t)(P)$ is amenable relative to $M^r \ovt \theta_t(B)$ inside $M^r \ovt \Mtil$. Since $\theta_t(B) = B \subset N$, we get a fortiori that $(\id \ot \theta_t)(P)$ is amenable relative to $M^r \ovt N$. By \cite[Theorem 1.6]{PV11}, we get that either $(\id \ot \theta_t)(P) \prec M^r \ovt N$, or that $(\id \ot \theta_t)\Delta(M_0)$ is amenable relative to $M^r \ovt N$. So we are in one of the following situations.

{\bf Case 1.} There exists a $t \in (0,1)$ such that $(\id \ot \theta_t)(P) \prec M^r \ovt N$. Using Proposition \ref{prop.comult}.\ref{comult-two} and assumption \ref{assum.embed}, we know that for all $j \in \{1,2\}$, we have $\Delta(M_0) \not\prec M^r \ovt (B \rtimes \Gamma_j)$. Since $M^r \ovt M$ is the amalgamated free product of $M^r \ovt (B \rtimes \Gamma_1)$ and $M^r \ovt (B \rtimes \Gamma_2)$, amalgamated over $M^r \ovt B$, it follows from \cite[Theorem 3.2]{Io12b} that $P \prec M^r \ovt B$.

{\bf Case 2.} For all $t \in (0,1)$, we have that $(\id \ot \theta_t)\Delta(M_0)$ is amenable relative to $M^r \ovt N$. Fix $t \in (0,1)$. We get that $\Delta(M_0)$ is amenable relative to $M^r \ovt \theta_{-t}(N)$ inside $M^r \ovt \Mtil$. By Proposition \ref{prop.comult}.\ref{comult-three}, we conclude that $M_0$ is amenable relative to $\theta_{-t}(N)$ inside $\Mtil$. Hence $\theta_t(M_0)$ is amenable relative to $N$ inside $\Mtil$. So $\theta_t(M_0)$ is amenable relative to $N$ inside $\Mtil$ for all $t \in (0,1)$. Also it is part of assumption \ref{assum.group}.\ref{assum.group-two} that $M_0' \cap M^\omega = \C1$. Viewing $M$ as the amalgamated free product of $B \rtimes \Gamma_1$ and $B \rtimes \Gamma_2$, amalgamated over $B$, it follows from \cite[Theorem 5.1]{Io12b} that either $M_0 \prec B \rtimes \Gamma_j$ for some $j \in \{1,2\}$, or that $M_0$ is amenable relative to $B$ inside $M$. The first option contradicts assumption \ref{assum.embed}, while the second option contradicts assumption \ref{assum.rel-amen}. Hence case 2 is ruled out.
\end{proof}

\section{Left-right wreath products and inner amenability}

We need the following elementary results on left-right wreath products $H^{(\Gamma)} \rtimes (\Gamma \times \Gamma)$, where the direct product group $\Gamma \times \Gamma$ acts on the set $\Gamma$ by left-right multiplication: $(g,h) \cdot k = gkh^{-1}$. We refer to Section \ref{sec.inner-amen} for the definition of inner amenability.

\begin{proposition}\label{prop.left-right-wreath}
Let $H$ and $\Gamma$ be arbitrary countable groups with $H \neq \{e\}$. Write $\cH := H^{(\Gamma)}$ and consider the left-right wreath product $\cG := \cH \rtimes (\Gamma \times \Gamma)$. Denote by $H_1$ the abelianization of $H$ with quotient map $p_1 : H \recht H_1$. Define the homomorphism
$$p : \cH \recht H_1 : p(x) = \sum_{g \in \Gamma} p_1(x_g) \; .$$
Denote by $\cH_0$ the kernel of $p$ and define $\cG_0 := \cH_0 \rtimes (\Gamma \times \Gamma)$.
\begin{enumerate}[label=(\alph*)]
\item\label{left-right-a} If $\Gamma$ is not inner amenable, also $\cG$ is not inner amenable. Even more so, the unitary representation $(\Ad g)_{g \in \Gamma \times \Gamma}$ on $\ell^2(\cG - \{e\})$ does not have almost invariant vectors. So any subgroup of $\cG$ that contains $\Gamma \times \Gamma$ is not inner amenable.
\item\label{left-right-b} If $\Gamma$ is nonamenable and finitely generated and if $\Gamma$ has trivial center, then $\cG$ is not inner amenable.
\item\label{left-right-c} If $\Gamma$ is infinite and has trivial center, then $\cG_0$ and $\cG$ are icc groups and $$(L\cH_0)' \cap L \cG \subset L \cH \quad\text{and}\quad (L\cG_0)' \cap L \cG = \C 1 \; .$$
\end{enumerate}
\end{proposition}

Statement \ref{left-right-b} in the above proposition is not used in the paper. We added it in order to put it in contrast with Remark \ref{rem.McDuff}, where we show that there are nonamenable icc groups $\Gamma$ such that $L\cG$ is a McDuff II$_1$ factor, and in particular such that $\cG$ is not W$^*$-superrigid.

\begin{proof}
Throughout the proof, we write $G := \Gamma \times \Gamma$. We denote by $P_G$ the orthogonal projection of $\ell^2(\cG)$ onto $\ell^2(G)$. The action $(\Ad g)_{g \in \Gamma \times \{e\}}$ on $\cG - G$ has finite stabilizers. Therefore, the restriction of the representation $(\Ad g)_{g \in \Gamma \times \{e\}}$ to the invariant subspace $\ell^2(\cG - G)$ is weakly contained in the regular representation of $\Gamma$.

\ref{left-right-a}\ Assume that $\xi_n \in \ell^2(\cG - \{e\})$ is a sequence of vectors that is almost invariant under $(\Ad g)_{g \in G}$. By the remark in the first paragraph and because $\Gamma$ is nonamenable, it follows that $\|\xi_n - P_G(\xi_n)\|_2 \recht 0$. Note that $(P_G(\xi_n))$ is a sequence of vectors in $\ell^2(G - \{e\})$ that is almost invariant under $(\Ad_g)_{g \in G}$. Since $\Gamma$ is not inner amenable, also $G$ is not inner amenable. Hence $\|P_G(\xi_n)\|_2 \recht 0$. So also $\|\xi_n\|_2 \recht 0$.

\ref{left-right-b}\ Assume that $\xi_n \in \ell^2(\cG - \{e\})$ is a sequence of vectors that is almost invariant under $(\Ad g)_{g \in \cG}$. By the remark in the first paragraph and because $\Gamma$ is nonamenable, it follows that $\|\xi_n - P_G(\xi_n)\|_2 \recht 0$.
Fix an element $s \in H-\{e\}$. For every $k \in \Gamma$, denote by $s_k \in H^{(\Gamma)}$ the element $s$ viewed in position $k$. It is easy to check that $P_G \circ (\Ad s_k) \circ P_G = P_{\Stab k}$. Since $\|\xi_n - P_G(\xi_n)\| \recht 0$ and since the sequence $(\xi_n)$ is almost invariant under $(\Ad g)_{g \in \cG}$, we conclude that $\|\xi_n - P_{\Stab \cF}(\xi_n)\| \recht 0$ for every finite subset $\cF \subset \Gamma$. If $\{k_1,\ldots,k_r\}$ is a finite generating set for $\Gamma$, one checks that $\Stab \{e,k_1,\ldots,k_r\} = \{(g,g) \mid g \in \operatorname{Center} \Gamma \}$. Since $\Gamma$ has trivial center, we get that $\|\xi_n\| \recht 0$.

\ref{left-right-c}\ We start by proving the following claim: for every $g \in G-\{e\}$, there exist infinitely many $k \in \Gamma$ such that $g \cdot k \neq k$. To prove this claim, denote $\delta : \Gamma \recht G : \delta(h) = (h,h)$. If $g \in G$ is not conjugate with an element in $\delta(\Gamma)$, we have $g \cdot k \neq k$ for all $k \in \Gamma$ and the claim is trivial. If $g \in G - \{e\}$ is conjugate with the element $\delta(h) \in \delta(\Gamma)$, we may actually assume that $g = \delta(h)$. Given $k \in \Gamma$, we have $g \cdot k = k$ if and only if $h$ commutes with $k$. So if $g \cdot k = k$ for all but finitely many $k \in \Gamma$, it follows that the centralizer $\Gamma_0 := \operatorname{Centr}_\Gamma(h)$ of $h$ inside $\Gamma$ has a finite complement. Since $\Gamma_0 < \Gamma$ is a subgroup and $\Gamma$ is infinite, this implies that $\Gamma_0 = \Gamma$. So $h$ lies in the center of $\Gamma$. This is impossible, because we assumed that $\Gamma$ has trivial center and that $\delta(h) = g \neq e$.

Having proven the claim above, we show that for every $x \in \cG - \cH$, we have that $\{z x z^{-1} \mid z \in \cH_0\}$ is infinite. We write $x = y g$ with $y \in \cH$ and $g \in G - \{e\}$. Define $$\cF_0 := \{e\} \cup \{g \cdot e\} \cup \{k \in \Gamma \mid y_k \neq e\} \; .$$
By the claim in the previous paragraph, we can inductively choose elements $k_n \in \Gamma$ such that $g \cdot k_n \neq k_n$ for all $n$ and such that the sets $\cF_0$, $\{k_1,g \cdot k_1\}$, $\{k_2,g \cdot k_2\}$, ... are all disjoint. Fix an element $s \in H - \{e\}$. For every $k \in \Gamma$, denote by $s_k \in H^{(\Gamma)}$ the element $s$ viewed in position $k$. Define the sequence of elements $z_n \in \cH_0$ given by $z_n := s_e^{-1} s_{k_n}$. Since
$$z_n x z_n^{-1}  = s_e^{-1} \, s_{k_n} \, y \, s_{g \cdot k_n}^{-1} \, s_{g \cdot e} \; g \; ,$$
we get that all elements $z_n x z_n^{-1}$ are distinct. So the set $\{z x z^{-1} \mid z \in \cH_0\}$ is infinite for every $x \in \cG - \cH$. This means that $(L\cH_0)' \cap L \cG \subset L \cH$.

It remains to prove that $(L \cG_0)' \cap L \cG = \C 1$. Because of the previous paragraph, it suffices to observe that elements in $\cH - \{e\}$ have an infinite conjugacy class under $(\Ad g)_{g \in \Gamma \times \{e\}}$.
\end{proof}

\begin{remark} \label{rem.McDuff}
There are nonamenable icc groups $\Gamma$ such that $\cG := H^{(\Gamma)} \rtimes (\Gamma \times \Gamma)$ is inner amenable, and even such that $L \cG$ is a McDuff II$_1$ factor (see Section \ref{sec.inner-amen} for terminology). Indeed, it suffices that $\Gamma$ admits two sequences of elements $(g_n),(h_n)$ with the property that $g_n$ and $h_n$ do not commute, but eventually commute with any fixed element of $\Gamma$. In that case, $u_{(g_n,g_n)}$ and $u_{(h_n,h_n)}$ form two noncommuting central sequences in $L \cG$, forcing $L \cG$ to be McDuff. Such sequences can be easily found in the icc group $S_\infty$ of finite permutations of $\N$, and hence also in the nonamenable icc group $\F_2 \times S_\infty$.

Because of the previous paragraph, \emph{not all nonamenable left-right wreath product groups are W$^*$-superrigid.}
\end{remark}

\section{Comultiplications and relative commutants}\label{sec.comult-rel-commutant}

\begin{lemma}\label{lem.rel-icc-intertwine}
Let $G$ and $\cG$ be countable groups and $\gamma_i : G \recht \cG$ group homomorphisms, with $i=1,2$. Assume that for every $h \in \cG - \{e\}$, the set $\{\gamma_1(g) h \gamma_1(g)^{-1} \mid g \in G \}$ is infinite. Then the following statements are equivalent.
\begin{enumerate}[label=(\alph*)]
\item\label{comult.aaa} There exists an $h \in \cG$ such that $\gamma_1(g) = h \gamma_2(g) h^{-1}$ for all $g \in G$.
\item\label{comult.bbb} There exists a finite subset $\cF \subset \cG$ such that $\cF \cap \gamma_1(g) \cF \gamma_2(g)^{-1} \neq \emptyset$ for all $g \in G$.
\item\label{comult.ccc} The unitary representation
$$\pi : G \recht \cU(\ell^2 \cG) : \pi(g)\xi = u_{\gamma_1(g)} \xi u_{\gamma_2(g)}^*$$
is not weakly mixing.
\end{enumerate}
\end{lemma}
\begin{proof}
The equivalence of \ref{comult.bbb} and \ref{comult.ccc} follows from Lemma \ref{lem.prel-weak-mixing}. The implication \ref{comult.aaa} $\Rightarrow$ \ref{comult.bbb} is trivial by taking $\cF = \{h\}$. Conversely assume that \ref{comult.bbb} holds. By Lemma \ref{lem.prel-weak-mixing}, we can take an $h \in \cG$ such that $\cF_1 := \{\gamma_1(g) h \gamma_2(g)^{-1} \mid g \in G\}$ is a finite set. It follows that $\cF_1 \cF_1^{-1}$ is a finite subset of $\cG$ that is globally invariant under $(\Ad \gamma_1(g))_{g \in G}$. By our assumptions, it follows that $\cF_1 \cF_1^{-1} = \{e\}$. This means that $\cF_1$ is a singleton. So $\cF_1 = \{h\}$ and we conclude that $\gamma_1(g) = h \gamma_2(g) h^{-1}$ for all $g \in G$.
\end{proof}

\begin{lemma}\label{lem.funny}
Let $\Lambda$ be an icc group and $\al,\beta \in \Aut(L \Lambda)$. Denote by $(v_s)_{s \in \Lambda}$ the canonical group of unitaries generating $L \Lambda$. Let $\Delta : L \Lambda \recht L \Lambda \ovt L \Lambda : \Delta(v_s) = v_s \ot v_s$ be the comultiplication. If $(\al \ot \beta)\Delta(L \Lambda) \prec \Delta(L \Lambda)$, there exist unitaries $V,W \in L \Lambda$, characters $\om,\mu : \Lambda \recht \T$ and an automorphism $\delta \in \Aut(\Lambda)$ such that $\al(v_s) = \om(s) \, V v_{\delta(s)} V^*$ and $\beta(v_s) = \mu(s) \, W v_{\delta(s)} W^*$ for all $s \in \Lambda$.
\end{lemma}
\begin{proof}
We start by proving the following claim: if $\Lambda$ admits a sequence of elements $s_n \in \Lambda$ such that
\begin{equation}\label{eq.assum-by-contra}
\lim_n \|E_{\Delta(L \Lambda)}( v_x^* \al(v_{s_n}) \ot \beta(v_{s_n}) v_y^*) \|_2 = 0 \quad\text{for all}\;\; x,y \in \Lambda \; ,
\end{equation}
then $(\al \ot \beta)\Delta(L \Lambda) \not\prec \Delta(L \Lambda)$. Indeed, if \eqref{eq.assum-by-contra} holds, we multiply left and right by elements of the form $\Delta(v_a)$, $\Delta(v_b)$ and conclude that
$$\lim_n \|E_{\Delta(L \Lambda)}\bigl( (v_x \ot v_y) \, (\al \ot \beta)\Delta(v_{s_n}) \, (v_a \ot v_b)\bigr) \|_2 = 0 \quad\text{for all}\;\; x,y,a,b, \in \Lambda \; .$$
Using $\|\,\cdot\,\|_2$-approximations, it follows that the same holds when we replace $v_x \ot v_y$ and $v_a \ot v_b$ by arbitrary elements of $L \Lambda \ovt L \Lambda$. This then means that $(\al \ot \beta)\Delta(L \Lambda) \not\prec \Delta(L \Lambda)$ and hence this proves the claim.

Our assumption is that $(\al \ot \beta)\Delta(L \Lambda) \prec \Delta(L \Lambda)$. So by the claim above, there is no sequence of elements $s_n \in \Lambda$ satisfying \eqref{eq.assum-by-contra}. This means that there are finitely many $x_i,y_i \in \Lambda$, with $i=1,\ldots,k$, and a $\delta > 0$ such that
$$
\sum_{i=1}^k \|E_{\Delta(L \Lambda)}( v_{x_i}^* \al(v_{s}) \ot \beta(v_{s}) v_{y_i}^*) \|_2^2 \geq \delta \quad\text{for all}\;\; s \in \Lambda \; .
$$
The left hand side can be computed and we conclude that
\begin{equation}\label{eq.good-inequ}
\sum_{i=1}^k \sum_{t \in \Lambda} |\tau(v_{x_i t}^* \al(v_s))|^2 \, |\tau(v_{t y_i}^* \beta(v_s))|^2 \geq \delta \quad\text{for all}\;\; s \in \Lambda \; .
\end{equation}
As in \cite[Formula (3.1)]{IPV10}, we define the height of an element $a \in L\Lambda$ as
$$h_\Lambda(a) := \max \{ |\tau(v_t^* a)| \mid t \in \Lambda \} \; .$$
Using \eqref{eq.good-inequ}, we find that for all $s \in \Lambda$, we have
\begin{align*}
\delta &\leq \sum_{i=1}^k \sum_{t \in \Lambda} |\tau(v_{x_i t}^* \al(v_s))|^2 \, |\tau(v_{t y_i}^* \beta(v_s))|^2 \\
&\leq h_\Lambda(\al(v_s))^2 \; \sum_{i=1}^k \sum_{t \in \Lambda} |\tau(v_{t y_i}^* \beta(v_s))|^2 \\
&= k \; h_\Lambda(\al(v_s))^2 \; .
\end{align*}
So we get that $h_\Lambda(\al(v_s)) \geq \sqrt{\delta/k}$ for all $s \in \Lambda$. It then follows from \cite[Theorem 3.1]{IPV10} that there exist a unitary $V \in L \Lambda$, a character $\om : \Lambda \recht \T$ and an automorphism $\delta_1 \in \Aut(\Lambda)$ such that $\al(v_s) = \om(s) \, V v_{\delta_1(s)} V^*$ for all $s \in \Lambda$.

By symmetry, we find the same description of the automorphism $\beta$, yielding a unitary $W \in L \Lambda$, a character $\mu : \Lambda \recht \T$ and an automorphism $\delta_2 \in \Aut(\Lambda)$ such that $\beta(v_s) = \mu(s) \, W v_{\delta_2(s)} W^*$ for all $s \in \Lambda$. It remains to prove that up to an inner conjugacy, $\delta_1 = \delta_2$. Replacing $\al$ by $(\Ad V^*) \circ \al$ and replacing $\beta$ by $(\Ad W^*) \circ \beta$, we still have that $(\al \ot \beta)\Delta(L \Lambda) \prec \Delta(L \Lambda)$. So there exist finitely many $x_i,y_i \in \Lambda$, with $i=1,\ldots,k$, and a $\delta > 0$ such that \eqref{eq.good-inequ} holds. Since now $\al(v_s) = \om(s) \, v_{\delta_1(s)}$ and $\beta(v_s) = \mu(s) \, v_{\delta_2(s)}$, the left hand side of \eqref{eq.good-inequ} is zero, unless there exists an $i \in \{1,\ldots,k\}$ and a $t \in \Lambda$ satisfying $\delta_1(s) = x_i t$ and $\delta_2(s) = t y_i$. This means that for every $s \in \Lambda$, there exists an $i \in \{1,\ldots,k\}$ such that $\delta_1(s) y_i \delta_2(s)^{-1} = x_i$. Since $\Lambda$ is icc, it then follows from Lemma \ref{lem.rel-icc-intertwine} that $\delta_1$ and $\delta_2$ are equal up to inner conjugacy.
\end{proof}

Let $\Lambda$ be an icc group and assume that $L \Lambda$ does not have property Gamma, so that $\Out(L \Lambda)$ is a Polish group (see Section \ref{sec.inner-amen} for notations and terminology). For every character $\om \in \Lambdah$, we denote by $\al_\om$ the automorphism of $L \Lambda$ given by $\al_\om(v_s) = \om(s) \, v_s$ for all $s \in \Lambda$. Using the icc property, one checks that the map $\om \mapsto \al_\om$ embeds $\Lambdah$ continuously into $\Out(L \Lambda)$. Since $\Lambdah$ is compact, we can thus view $\Lambdah$ as a compact subgroup of $\Out(L \Lambda)$.

A countable subgroup $\cA$ of a Polish group $\cB$ is said to be \emph{discrete} if there exists a neighborhood $\cU$ of the identity $e$ in $\cB$ such that $\cU \cap \cA = \{e\}$.

\begin{lemma}\label{lem.comult-rel-commutant-bimod}
Let $M_0$ be a II$_1$ factor without property Gamma. Let $r > 0$ and $M_0^r = L \Lambda$ for some countable group $\Lambda$. Denote by $\Delta : M_0 \recht (M_0 \ovt M_0)^r$ the amplified comultiplication as in Section \ref{sec.comult}.

Assume that $M$ is a tracial von Neumann algebra with $M_0 \subset M$ and $M_0' \cap M = \C 1$. Let $\cL \subset \cN_M(M_0)$ be a subgroup such that $M = (M_0 \cup \cL)\dpr$. Finally assume that the image of $\cL$ in $\Out(M_0)$ is a discrete torsion-free subgroup. Then the following holds.
\begin{enumerate}[label=(\alph*)]
\item\label{rel-comm-bimod-one} If $\cH \subset L^2((M \ovt M)^r) \ominus L^2(\Delta(M_0))$ is a nonzero $\Delta(M_0)$-$\Delta(M_0)$-subbimodule of finite left $\Delta(M_0)$-dimension, then there exist automorphisms $\beta_1,\ldots,\beta_k \in \Aut(M_0)$ and a unitary $\psi : \cH \recht L^2(M_0)^{\oplus k} : \xi \mapsto (\psi_1(\xi),\ldots,\psi_k(x))$ such that
    $$\psi_i(\Delta(x) \, \xi \, \Delta(y)) =  x \, \psi_i(\xi) \, \beta_i(y) \quad\text{for all}\;\; x,y \in M_0, \xi \in \cH, i = 1,\ldots,k,$$
    and such that every $\beta_i$ generates a discrete infinite subgroup of $\Out(M_0)$.
\item\label{rel-comm-bimod-two} We have $\Delta(M_0)' \cap (M \ovt M)^r = \C 1$.
\end{enumerate}
\end{lemma}

Note that the setup of Lemma \ref{lem.comult-rel-commutant-bimod} would allow to write $M$ as the cocycle crossed product of $M_0$ and an outer cocycle action of $\cL / (\cL \cap \Inn M_0)$ on $M_0$, but we do not need that formalism here.

\begin{proof}
First note that statement \ref{rel-comm-bimod-two} is a consequence of statement \ref{rel-comm-bimod-one}. Take an element $T$ in $\Delta(M_0)' \cap (M \ovt M)^r$ and write $S := T - E_{\Delta(M_0)}(T)$. Since $M_0$ is a factor, it suffices to prove that $S = 0$. So assume that $S \neq 0$. Denote by $\cH$ the closure of $\Delta(M_0) S$. Then $\cH$ is a $\Delta(M_0)$-$\Delta(M_0)$-subbimodule of $L^2((M \ovt M)^r) \ominus L^2(\Delta(M_0))$ that has finite left dimension. By construction $\cH$ contains the nonzero vector $S$ satisfying $\Delta(x) S = S \Delta(x)$ for all $x \in M_0$. Write $\cH$ as in \ref{rel-comm-bimod-one}. Since all automorphisms $\beta_i$ are outer, we have that $\psi_i(S) = 0$ for all $i \in \{1,\ldots,k\}$. So $S = 0$, contradicting our assumption.

We now start proving statement \ref{rel-comm-bimod-one}.
Take a projection $p \in \M_n(\C) \ot M_0$ with $(\Tr \ot \tau)(p) = r$. Realize $M_0^r := p(\M_n(\C) \ot M_0)p$ and $(M_0 \ovt M_0)^r = M_0^r \ovt M_0$. Denote by $\Delta : L \Lambda \recht L \Lambda \ovt L \Lambda : \Delta(v_s) = v_s \ot v_s$ the original comultiplication. During the proof, to improve the clarity of the exposition, we denote the amplified comultiplication by $\Delta_0 : M_0 \recht M_0^r \ovt M_0$. The relation between $\Delta_0$ and $\Delta$ has been concretized in Remark \ref{rem.amplify-explicit}.

Put $M^r := p(\M_n(\C) \ot M)p$ so that literally $M_0^r \subset M^r$. Let $\cH \subset L^2(M^r \ovt M) \ominus L^2(\Delta_0(M_0))$ be a $\Delta_0(M_0)$-$\Delta_0(M_0)$-subbimodule of finite left $\Delta_0(M_0)$-dimension. Using the notation of Remark \ref{rem.amplify-explicit}, we put
$$\cH' := Z \, (\id \ot \id \ot \zeta^{-1})(\cH \ot \M_n(\C)) \, Z^*$$
and notice that $\cH' \subset L^2(M^r \ovt M^r) \ominus L^2(\Delta(M_0^r))$ is a nonzero $\Delta(M_0^r)$-$\Delta(M_0^r)$-subbimodule of finite left $\Delta(M_0^r)$-dimension. To conclude the proof of the lemma, we have to find automorphisms $\beta_1,\ldots,\beta_k \in \Aut(M_0^r)$ and a unitary $\psi : \cH' \recht L^2(M_0^r)^{\oplus k} : \xi \mapsto (\psi_1(\xi),\ldots,\psi_k(x))$ such that
$$\psi_i(\Delta(x) \, \xi \, \Delta(y)) =  x \, \psi_i(\xi) \, \beta_i(y) \quad\text{for all}\;\; x,y \in M_0^r, \xi \in \cH', i = 1,\ldots,k,$$
and such that every $\beta_i$ generates an infinite discrete subgroup of $\Out(M_0^r)$.

By our assumptions on $M_0 \subset M$, we can choose a subset $\cL_0 \subset \cN_M(M_0)$ such that
$$L^2(M) = L^2(M_0) \oplus \bigoplus_{V \in \cL_0} L^2(M_0) V$$
and such that for every $V \in \cL_0$, the automorphism $\Ad V$ of $M_0$ generates a discrete infinite subgroup of $\Out(M_0)$. Fix $V \in \cL_0$. Take a partial isometry $v \in \M_n(\C) \ot M_0$ such that $vv^* = p$ and $v^* v = (\id \ot \Ad V)(p)$. Write $V' := v(1 \ot V)$ and note that $V' \in \cN_{M^r}(M_0^r)$. As such, we find a subset $\cL_1 \subset \cN_{M^r}(M_0^r)$ such that
$$L^2(M^r) = L^2(M_0^r) \oplus \bigoplus_{V \in \cL_1} L^2(M_0^r) V$$
and such that for every $V \in \cL_1$, the automorphism $\Ad V$ of $M_0^r$ generates a discrete infinite subgroup of $\Out(M_0^r)$.

Define the subset $\cL_2 \subset \cN_{M^r \ovt M^r}(M_0^r \ovt M_0^r)$ given by
$$\cL_2 := \{ 1 \ot V \mid V \in \cL_1 \} \cup \{V \ot 1 \mid V \in \cL_1 \} \cup \{V_1 \ot V_2 \mid V_1,V_2 \in \cL_1 \} \; .$$
We get that
\begin{equation}\label{eq.orthog-decomp-W}
L^2(M^r \ovt M^r) = L^2(M_0^r \ovt M_0^r) \oplus \bigoplus_{W \in \cL_2} L^2(M_0^r \ovt M_0^r) W \; .
\end{equation}
Also for every $W \in \cL_2$, the automorphism $\Ad W$ of $M_0^r \ovt M_0^r$ is of the form $\al_W \ot \beta_W$, where at least one of the $\al_W$, $\beta_W$ generates a discrete infinite subgroup of $\Out(M_0^r)$.

Denote by $P_0$ the orthogonal projection of $L^2(M^r \ovt M^r)$ onto the closed subspace $L^2(M_0^r \ovt M_0^r)$ and define $\cH_0$ as the closure of $P_0(\cH')$. Then $\cH_0 \subset L^2(M_0^r \ovt M_0^r) \ominus L^2(\Delta(M_0^r))$ is a $\Delta(M_0^r)$-$\Delta(M_0^r)$-subbimodule of finite left dimension. By \cite[Proposition 7.2.3]{IPV10}, we get that $\cH_0 = \{0\}$.

For every $W \in \cL_2$, denote by $P_W$ the orthogonal projection of $L^2(M^r \ovt M^r)$ onto the closed subspace $L^2(M_0^r \ovt M_0^r) W$ and define
$$\vphi_W : L^2(M^r \ovt M^r) \recht L^2(M_0^r \ovt M_0^r) : \vphi_W(\xi) = P_W(\xi) W^* \; .$$
Since $W$ normalizes $M_0^r \ovt M_0^r$ and since \eqref{eq.orthog-decomp-W} is an orthogonal decomposition, we get that
$$\vphi_W( \Delta(x) \, \xi \, \Delta(y)) = \Delta(x) \, \vphi_W(\xi) \, (\al_W \ot \beta_W)\Delta(y)$$
for all $x,y \in M_0^r$ and all $\xi \in L^2(M^r \ovt M^r)$. Denote by $\cH_W$ the closure of $\vphi_W(\cH')$. Below we prove the following statement: if $\cH_W \neq \{0\}$, then there exists a unitary $\psi_W : \cH_W \recht L^2(M_0^r)$ and an automorphism $\gamma_W \in \Aut(M_0^r)$ such that
\begin{equation}\label{eq.equiv-beta}
\psi_W(\Delta(x) \, \xi \, (\al_W \ot \beta_W)\Delta(y)) = x \, \psi_W(\xi) \, \gamma_W(y)
\end{equation}
for all $x,y \in M_0^r$ and all $\xi \in \cH_W$, and such that $\gamma_W$ generates a discrete infinite subgroup of $\Out(M_0^r)$. For the moment, we assume that the statement is proven and deduce the lemma from it. Whenever $\cH_W \neq \{0\}$, we denote by $\cK_W$ the $M_0^r$-$M_0^r$-bimodule $L^2(M_0^r)$ with bimodule action $x \cdot \xi \cdot y = x \xi \gamma_W(y)$. Then $\psi_W \circ \vphi_W : \cH' \recht \cK_W$ is a bimodular map with dense range. So, $\cK_W$ is isomorphic with a subbimodule of $\cH'$. Since $\cH'$ has finite left dimension and since $\cH_0 = \{0\}$, it follows that $\cH'$ is isomorphic with the direct sum of finitely many $\cK_W$'s. This proves the lemma.

So it remains to prove the statement above. Assume that $\cH_W \neq \{0\}$. By construction, $\cH_W$ is a $\Delta(M_0^r)$-$(\al_W \ot \beta_W)\Delta(M_0^r)$-subbimodule of $L^2(M_0^r \ovt M_0^r)$ of finite left dimension. By Theorem \ref{thm.intertwining}, this means that $(\al_W \ot \beta_W)\Delta(M_0^r) \prec \Delta(M_0^r)$. By Lemma \ref{lem.funny}, there exist characters $\om,\mu : \Lambda \recht \T$ and an automorphism $\delta \in \Aut(\Lambda)$ such that, after unitarily conjugating $\al_W$ and $\beta_W$, we have that $\al_W(v_s) = \om(s) \, v_{\delta(s)}$ and $\beta_W(v_s) = \mu(s) \, v_{\delta(s)}$ for all $s \in \Lambda$. Note that $(\al_W \ot \beta_W)\Delta(v_s) = \Delta(\gamma_W(v_s))$, where the automorphism $\gamma_W \in \Aut(M_0^r)$ is defined by the formula $\gamma_W(v_s) = \om(s) \mu(s) \, v_{\delta(s)}$.

So $(\al_W \ot \beta_W)\Delta(M_0^r) = \Delta(M_0^r)$. We get in particular that $\cH_W$ is a nonzero $\Delta(M_0^r)$-$\Delta(M_0^r)$-subbimodule of $L^2(M_0^r \ovt M_0^r)$ that has finite left dimension. It then follows from \cite[Proposition 7.2.3]{IPV10} that $\cH_W \subset L^2(\Delta(M_0^r))$. Since $M_0^r$ is a factor and $\cH_W \neq \{0\}$, we get that $\cH_W = L^2(\Delta(M_0^r))$. We can thus define $\psi : \cH_W \recht L^2(M_0^r)$ as being $\Delta^{-1}$. By construction, \eqref{eq.equiv-beta} holds. It remains to prove that $\gamma_W$ generates an infinite discrete subgroup of $\Out(M_0^r)$.

We know that at least one of the $\al_W$, $\beta_W$ generates an infinite discrete subgroup of $\Out(M_0^r)$. Assume that this is the case for $\al_W$. View $\al_W$ as an element of $\Out(M_0^r)$ and view $\Lambdah$ as a compact subgroup of $\Out(M_0^r)$. Since $\al_W(v_s) = \om(s) \, v_{\delta(s)}$ for all $s \in \Lambda$, we have that $\al_W$ normalizes $\Lambdah$. Since $\Lambdah$ is compact and since $\al_W$ generates an infinite discrete subgroup, it follows that $\Lambdah$ and $\al_W$ together generate a copy of $\Lambdah \rtimes \Z$ as a closed subgroup of $\Out(M_0^r)$. Since $\gamma_W \in \al_W \Lambdah$, it then follows that also $\gamma_W$ generates an infinite discrete subgroup of $\Out(M_0^r)$.
\end{proof}

For later use, we end this section with yet another elementary lemma.

\begin{lemma}\label{lem.descr-aut}
Let $\Lambda$ be a countable group and $\Delta : L \Lambda \recht L \Lambda \ovt L \Lambda$ the comultiplication given by $\Delta(v_s) = v_s \ot v_s$ for all $s \in \Lambda$.
If $\al,\beta \in \Aut(L \Lambda)$ are automorphisms that satisfy $(\al \ot \id) \circ \Delta = \Delta \circ \beta$, then there exists a character $\om : \Lambda \recht \T$ such that $\al = \beta = \al_\om$, where $\al_\om(v_s) = \om(s) v_s$ for all $s \in \Lambda$.
\end{lemma}
\begin{proof}
Since $\Delta(\beta(v_s)) = \al(v_s) \ot v_s$, we see that $\al(v_s) \ot v_s \in \Delta(L \Lambda)$. This implies
that $\al(v_s)$ must be a multiple of $v_s$, for all $s \in \Lambda$. So we find a character $\om : \Lambda \recht \T$ such that $\al = \al_\om$. But then also $\beta = \al_\om$.
\end{proof}

\section{Proof of Theorem \ref{thm.main-intro}} \label{sec.proof-main-thm}

Theorem \ref{thm.main-intro} will be a direct consequence of the following general result. Recall from Section \ref{sec.class-S-wa} the notions of weak amenability and class $\cS$.

\begin{theorem}\label{thm.main}
Let $\Gamma$ be a countable group satisfying one of the following conditions.
\begin{enumerate}
\item $\Gamma$ is nonamenable, icc, weakly amenable, belongs to class $\cS$ and admits a bound on the orders of its finite subgroups.
\item $\Gamma = \Gamma_1 * \Gamma_2$ with $|\Gamma_1| \geq 2$ and $|\Gamma_2| \geq 3$.
\end{enumerate}
Let $H$ be a nontrivial abelian group with subgroup $H_0 < H$. Assume that $H/H_0$ is either trivial or torsion-free. Define $\cH := H^{(\Gamma)}$ and consider the homomorphism
$$p_H : H^{(\Gamma)} \recht H : p_H(x) = \sum_{g \in \Gamma} x_g \; .$$
Denote $\cH_0 := p_H^{-1}(H_0)$ and $\cG_0 := \cH_0 \rtimes (\Gamma \times \Gamma)$.

If $\Lambda$ is any countable group and $\pi : L \Lambda \recht (L \cG_0)^r$ is a $*$-isomorphism for some $r > 0$, then $r = 1$ and $\Lambda \cong p_{H'}^{-1}(H_0') \rtimes (\Gamma \times \Gamma)$ for some abelian group $H'$ with subgroup $H_0' < H'$ such that $|H_0| = |H_0'|$ and $H/H_0 \cong H'/H'_0$.

More precisely, there exist group isomorphisms $\delta : \Lambda \recht p_{H'}^{-1}(H_0') \rtimes (\Gamma \times \Gamma)$ and $\gamma : H'/H'_0 \recht H/H_0$, a probability measure preserving isomorphism $\theta : \widehat{H'} \recht \widehat{H}$ satisfying $\theta(k + \eta) = \gammah(k) + \theta(\eta)$ for all $k \in \widehat{H'/H'_0}$ and a.e.\ $\eta \in \widehat{H'}$, a character $\om : \cG_0 \recht \T$ and a unitary $w \in L \cG_0$ such that $\pi = \Ad w \circ \al_\om \circ \pi_\theta \circ \pi_\delta$ where
\begin{itemize}
\item $\pi_\delta : L \Lambda \recht L\bigl(p_{H'}^{-1}(H_0') \rtimes (\Gamma \times \Gamma)\bigr)$ is the $*$-isomorphism given by $\pi_\delta(v_s) = u_{\delta(s)}$ for all $s \in \Lambda$~;
\item $\pi_\theta : L\bigl(p_{H'}^{-1}(H_0') \rtimes (\Gamma \times \Gamma)\bigr) \recht L\bigl(p_H^{-1}(H_0) \rtimes (\Gamma \times \Gamma)\bigr)$ is the natural $*$-isomorphism associated with an infinite tensor product of copies of $\theta$~;
\item $\al_\om$ is the automorphism of $L \cG_0$ given by $\al_\om(u_g) = \om(g) u_g$ for all $g \in \cG_0$.
\end{itemize}
\end{theorem}

This whole section is devoted to the proof of Theorem \ref{thm.main}, following closely the strategy of \cite{IPV10} and using many results of \cite{IPV10}. At the end, we will deduce Theorem \ref{thm.main-intro}, with case \ref{thm.main-intro}.1 corresponding to the special case where $H_0 = H$, and case \ref{thm.main-intro}.2 corresponding to $H_0 = \{e\}$.

Throughout this section, we fix a countable icc group $\Gamma$ that satisfies either condition 1 or condition 2 in Theorem \ref{thm.main}. We also fix a
nontrivial abelian group $H$ with subgroup $H_0 < H$ such that $H/H_0$ is either trivial or torsion-free. We denote $\cH := H^{(\Gamma)}$ and $\cH_0 := p_H^{-1}(H_0)$.
We write $G := \Gamma \times \Gamma$  and we consider the left-right wreath product $\cG := \cH \rtimes G$, with its subgroup $\cG_0 := \cH_0 \rtimes G$.
Put $M := L \cG$ and $M_0 := L \cG_0$.

We finally fix a countable group $\Lambda$, a positive number $r > 0$ and a $*$-isomorphism $\pi : L \Lambda \recht M_0^r$. To simplify notations, we do not explicitly write $\pi$ and identify $M_0^r = L \Lambda$.

As in Section \ref{sec.comult}, we consider the amplified comultiplication $\Delta : M_0 \recht (M_0 \ovt M_0)^r$. Note that the amplified homomorphism $\Delta$ is only defined up to unitary conjugacy (see Remark \ref{rem.amplify-explicit} for details).

Both condition 1 and condition 2 in Theorem \ref{thm.main} guarantee that $\Gamma$ is not inner amenable. So by Proposition \ref{prop.left-right-wreath}.\ref{left-right-c}, $\cG_0$ and $\cG$ are icc groups, $M_0$ and $M$ are II$_1$ factors and $M_0' \cap M^\omega = \C 1$. Also $M_0$ does not have property Gamma and $\Out(M_0)$ is a Polish group.

We write $A := L \cH$ so that $M = A \rtimes G$. We also write $A_0 := L \cH_0$ so that $M_0 = A_0 \rtimes G$.

Recall that a countable subgroup $\cA$ of a Polish group $\cB$ is said to be \emph{discrete} if there exists a neighborhood $\cU$ of the identity $e$ in $\cB$ such that $\cU \cap \cA = \{e\}$.

We start by two general lemmas on the structure of $M_0$ and $M$. The first one is an immediate consequence of Popa's cocycle superrigidity theorem \cite[Theorem 1.1]{Po06b}.

\begin{lemma}\label{lem.an-aut-lem}
Let $\beta \in \Aut(M_0)$ and assume that there exists a nonzero vector $\xi \in L^2(M_0)$ such that $\xi_0 \beta(a) = a \xi_0$ for all $a \in A_0$. Then there exists a character $\om : G \recht \T$ and a unitary $v \in \cN_M(M_0)$ such that $\beta = (\Ad v) \circ \al_\om$, where the automorphism $\al_\om$ is defined as $\al_\om(a u_g) = \om(g) \, a u_g$ for all $a \in A_0$, $g \in G$.

If moreover $\beta$ generates a discrete infinite subgroup of $\Out(M_0)$, we have that $E_{M_0}(v) = 0$.
\end{lemma}
\begin{proof}
Taking the polar decomposition of $\xi_0$, we find a nonzero partial isometry $v_0 \in M_0$ such that $v_0 \beta(a) = a v_0$ for all $a \in A_0$. By Proposition \ref{prop.left-right-wreath}.\ref{left-right-c}, we have that $A_0' \cap M_0 = A_0$ and hence, $v_0 v_0^* \in A_0$ and $v_0^* v_0 \in \beta(A_0)$. Since $G \actson A_0$ is ergodic, we can extend $v_0$ to a unitary $v_1 \in \cU(M_0)$ such that $v_1 \beta(a) v_1^* = a$ for all $a \in A_0$. Put $\beta_1 = (\Ad v_1) \circ \beta$. Since $\beta_1(a) = a$ for all $a \in A_0$, we have $\beta_1(u_g) = \mu_g \, u_g$ for all $g \in G$, where $\mu_g \in \cU(A_0)$ and $(\mu_g)_{g \in G}$ defines a $1$-cocycle for the action $G \actson A_0$.

By Popa's cocycle superrigidity theorem \cite[Theorem 1.1]{Po06b} for the action $G \overset{\si}{\actson} A$, we find a unitary $v_2 \in \cU(A)$ and a character $\om : G \recht \T$ such that $\mu_g = \om(g) \, v_2^* \si_g(v_2)$ for all $g \in G$. It follows that $v_2 \beta_1(x) v_2^* = \al_\om(x)$ for all $x \in M_0$. In particular, $v_2 \in \cN_M(M_0)$. Putting $v := v_1^* v_2^*$, we have that $v \in \cN_M(M_0)$ and $\beta = (\Ad v) \circ \al_\om$.

Finally, if $\beta$ generates a discrete infinite subgroup of $\Out(M_0)$, we know that as an element of $\Out(M_0)$, $\beta$ does not belong to the compact subgroup $\widehat{G} \subset \Out(M_0)$. So, $v \not\in \cU(M_0)$. Since $v \in \cN_M(M_0)$ and $M_0' \cap M = \C 1$, it follows that $E_{M_0}(v) = 0$.
\end{proof}

\begin{lemma}\label{lem.discrete-subgroup-Out}
Denote by $H_e$ the copy of $H$ inside $\cH$ in position $e \in \Gamma$. Then $M$ is generated by $M_0$ and the group of unitaries $\cL := \{u_s \mid s \in H_e\}$ that normalize $M_0$.
The image of $\cL$ in $\Out(M_0)$ is a discrete subgroup of $\Out(M_0)$ that is isomorphic with $H/H_0$.
\end{lemma}
\begin{proof}
By Proposition \ref{prop.left-right-wreath}.\ref{left-right-a}, we know that $M_0' \cap M^\om = \C 1$. So whenever $(a_n)$ is a sequence of unitaries in $\cU(M)$ satisfying $\|x a_n - a_n x\|_2 \recht 0$ for all $x \in M_0$, there exists a sequence $\lambda_n \in \T$ such that $\|a_n - \lambda_n 1 \|_2 \recht 0$. Assume that we have a sequence $s_n \in H_e$ such that $\Ad(u_{s_n})$, viewed as a sequence in $\Out(M_0)$, converges to the identity. We must prove that $s_n$ belongs to $(H_0)_e$ eventually. Since $\Ad(u_{s_n})$ converges to the identity in $\Out(M_0)$, we find a sequence of unitaries $w_n \in \cU(M_0)$ such that $\Ad(w_n u_{s_n}) \recht \id$ in $\Aut(M_0)$. This means that $\|x w_n u_{s_n} - w_n u_{s_n} x\|_2 \recht 0$ for all $x \in M_0$. It follows that we can take a sequence $\lambda_n \in \T$ such that $\|w_n u_{s_n} - \lambda_n 1 \|_2 \recht 0$. So $\|u_{s_n} - \lambda_n w_n^*\|_2 \recht 0$. In particular, we get that $\|u_{s_n} - E_{M_0}(u_{s_n})\|_2 \recht 0$. Since $\|u_{s_n} - E_{M_0}(u_{s_n})\|_2 = 1$ whenever $s_n \not\in (H_0)_e$, we conclude that $s_n \in (H_0)_e$ eventually.
\end{proof}

We now start a systematic study of the amplified comultiplication $\Delta : M_0 \recht (M_0 \ovt M_0)^r$.

\begin{lemma}\label{lem.triv-rel-comm}
We have that $\Delta(M_0)' \cap (M \ovt M)^r = \C 1$.
\end{lemma}
\begin{proof}
This is an immediate consequence of Lemma \ref{lem.discrete-subgroup-Out}, the assumption that $H/H_0$ is torsion-free and part \ref{rel-comm-bimod-two} of Lemma \ref{lem.comult-rel-commutant-bimod}.
\end{proof}

In what follows, we apply twice Theorem \ref{thm.norm-amen}. So we need to check that the assumptions stated in the beginning of Section \ref{sec.norm-amen} are satisfied.

\begin{lemma}\label{lem.assum-ok}
Both when viewing $M$ as the crossed product $M = B \rtimes (\{e\} \times \Gamma)$ with $B = A \rtimes (\Gamma \times \{e\})$, or as the crossed product $M = B \rtimes (\Gamma \times \{e\})$ with $B = A \rtimes (\{e\} \times \Gamma)$, all assumptions in the beginning of Section \ref{sec.norm-amen} are satisfied. More concretely, we have
\begin{enumerate}[label=(\alph*)]
\item\label{aaaa} $M_0' \cap M^\om = \C 1$,
\item\label{bbbb} $\Delta(M_0)' \cap (M \ovt M)^r = \C 1$,
\item\label{cccc} if $\Gamma_0 < \Gamma$ is a subgroup of infinite index,
$$M_0 \not\prec A \rtimes (\Gamma \times \Gamma_0) \quad\text{and}\quad M_0 \not\prec A \rtimes (\Gamma_0 \times \Gamma) \; ,$$
\item\label{dddd} $M_0$ is nonamenable relative to $A \rtimes (\Gamma \times \{e\})$, and nonamenable relative to $A \rtimes (\{e\} \times \Gamma)$, inside $M$.
\end{enumerate}
\end{lemma}
\begin{proof}
We already observed above that \ref{aaaa} follows from Proposition \ref{prop.left-right-wreath}.\ref{left-right-a}. Statement \ref{bbbb} is given by Lemma \ref{lem.triv-rel-comm}.
Statements \ref{cccc} is straightforward and statement \ref{dddd} follows from Lemma \ref{lem.rel-amen-subgroups}.
\end{proof}

\begin{lemma}\label{lem.embed-A}
We have $\Delta(A_0) \prec^f A \ovt A$.
\end{lemma}
\begin{proof}
Because of Lemma \ref{lem.assum-ok}, we can apply Theorem \ref{thm.norm-amen} to the crossed product decompositions $M = (A \rtimes (\Gamma \times \{e\})) \rtimes \Gamma$ and $M = (A \rtimes (\{e\} \times \Gamma)) \rtimes \Gamma$, and the abelian (hence amenable) von Neumann subalgebra $\Delta(A_0) \subset M^r \ovt M$. We conclude that
$$\Delta(A_0) \prec^f M^r \ovt \bigl(A \rtimes (\Gamma \times \{e\})\bigr) \quad{and}\quad \Delta(A_0) \prec^f M^r \ovt \bigl(A \rtimes (\{e\} \times \Gamma)\bigr) \; .$$
So by Lemma \ref{lem.full-embed}.\ref{full-two}, we get that $\Delta(A_0) \prec^f M^r \ovt A$. By symmetry, we also have that $\Delta(A_0) \prec^f A \ovt M^r$. Again by Lemma \ref{lem.full-embed}.\ref{full-two}, we conclude that $\Delta(A_0) \prec^f A \ovt A$.
\end{proof}

\begin{lemma}\label{lem.no-embed-Delta-LG}
Let $G_1 < G$ be a subgroup of infinite index. Then $\Delta(L G) \not\prec M^r \ovt (A \rtimes G_1)$ and $\Delta(L G) \not\prec (A \rtimes G_1) \ovt M^r$.
\end{lemma}
\begin{proof}
By symmetry, it suffices to prove that $\Delta(L G) \not\prec M^r \ovt (A \rtimes G_1)$. Assume the contrary. A combination of Lemma \ref{lem.embed-A} and Lemma \ref{lem.embed-with-normalizer} then gives that $\Delta(M_0) \prec M^r \ovt (A \rtimes G_1)$. Proposition \ref{prop.comult}.\ref{comult-two} now implies that $M_0 \prec A \rtimes G_1$, contradicting the assumption that $G_1 < G$ has infinite index.
\end{proof}

We can view $M$ as the generalized Bernoulli crossed product $M = (L H)^\Gamma \rtimes G$. As in Section~\ref{sec.spectral-gap}, we have the tensor length deformation by automorphisms $\al_t$ of the tracial von Neumann algebra $\Mtil := (LH * L\Z)^\Gamma \rtimes G$.

\begin{lemma}\label{lem.tussenstap}
Let $P \subset (M \ovt M)^r$ be a von Neumann subalgebra such that for all nonzero projections $p \in P' \cap (M \ovt M)^r$, we have that $P p$ is nonamenable relative to $M^r \ot 1$. Assume that $\Delta(L G) \subset \cN_{(M \ovt M)^r}(P)\dpr$. Then
$$\sup_{b \in \cU(P' \cap (M \ovt M)^r)} \|(\id \ot \al_t)(b) - b\|_2 \recht 0 \quad\text{as $t \recht 0$.}$$
\end{lemma}
\begin{proof}
We concretely realize the amplification $(M \ovt M)^r$ as $M^r \ovt M$.
Since $A$ is abelian and hence amenable, we have that $P p$ is nonamenable relative to $M^r \ovt A$, for all nonzero projections $p \in P' \cap M^r \ovt M$.

{\bf Case 1:} $\Gamma$ is a nonamenable group in class $\cS$ with the property that all finite subgroups of $\Gamma$ have order at most $\kappa-1$, for some fixed $\kappa \in \N$. We consider the left-right action $G \actson \Gamma$. We claim that $\Stab \cF$ is amenable whenever $\cF \subset \Gamma$ satisfies $|\cF| \geq \kappa$. Indeed, every $\Stab \cF$ is isomorphic with a subgroup of $\Gamma$ defined as the centralizer of $\kappa$ distinct elements. These $\kappa$ distinct elements necessarily generate an infinite subgroup of $\Gamma$. Since $\Gamma$ belongs to class $\cS$, the centralizer of an infinite subgroup is amenable (see Section \ref{sec.class-S-wa}). This proves the claim. So the conclusion of the lemma follows immediately from Theorem \ref{thm.spectral-gap-rigidity}, even without using the assumption that $\Delta(L G) \subset \cN_{M^r \ovt M}(P)\dpr$.

{\bf Case 2:} $\Gamma = \Gamma_1 * \Gamma_2$ with $|\Gamma_1| \geq 2$ and $|\Gamma_2| \geq 3$. Denote by $\delta : \Gamma \recht \Gamma \times \Gamma : \delta(h) = (h,h)$ the diagonal embedding and consider the left-right action $G \actson \Gamma$. Whenever $\cF \subset \Gamma$ and $|\cF| \geq 2$, we have that $\Stab \cF$ is either cyclic, or conjugate to a subgroup of $\delta(\Gamma_1)$, or conjugate to a subgroup of $\delta(\Gamma_2)$. So the conclusion of the lemma follows from Theorem \ref{thm.spectral-gap-rigidity}, once we have proven that $P p$ is nonamenable relative to $M^r \ovt (A \rtimes \delta(\Gamma_j))$ for all $j \in \{1,2\}$ and all nonzero projections $p \in P' \cap M^r \ovt M$.

By symmetry, it suffices to consider $j = 1$. Take a nonzero projection $p \in P' \cap M^r \ovt M$ and assume by contradiction that $P p$ is amenable relative to $M^r \ovt (A \rtimes \delta(\Gamma_1))$. Denote by $Q$ the normalizer of $P$ inside $M^r \ovt M$. By assumption, we have that $\Delta(L G) \subset Q$. Replacing $p$ by the smallest projection in $\cZ(Q)$ that dominates $p$ and using Lemma \ref{lem.rel-amen-max-proj}, we still have that $P p$ is amenable relative to $M^r \ovt (A \rtimes \delta(\Gamma_1))$.

We now prove that $P p$ is amenable relative to $B := M^r \ovt (A \rtimes (\Gamma \times \{e\}))$.
We denote $\cM_j := M^r \ovt (A \rtimes (\Gamma \times \Gamma_j))$. We can then view $M^r \ovt M$ as the amalgamated free product of $\cM_1$ and $\cM_2$ over $B$. Since we assumed that $Pp$ is amenable relative to $M^r \ovt (A \rtimes \delta(\Gamma_1))$, we have a fortiori that $P p$ is amenable relative to $\cM_1$. Since $p \in \cZ(Q)$, the normalizer of $P p$ inside $M^r \ovt M$ contains $Q p$. Since $\Delta(L G) \subset Q$ and since $\Gamma \times \Gamma_1$ has infinite index in $G$, it follows from Lemma \ref{lem.no-embed-Delta-LG} that $Q p \not\prec \cM_1$. Then \cite[Corollary 2.12]{Io12b} implies that $P p$ is amenable relative to $B$.

By symmetry, we also get that $P p$ is amenable relative to $M^r \ovt (A \rtimes (\{e\} \times \Gamma))$. So Lemma \ref{lem.rel-amen-intersect} implies that $P p$ is amenable relative to $M^r \ovt A$, and hence also relative to $M^r \ot 1$, contradicting our initial assumptions on $P$.
\end{proof}

\begin{lemma}\label{lem.conjugate-G}
There exists a unitary $\Om \in (M \ovt M)^r$ such that
$$\Om \; \Delta(LG) \; \Om^* \subset (LG \ovt LG)^r \; .$$
\end{lemma}
\begin{proof}
Also $M \ovt M$ can be viewed as a generalized Bernoulli crossed product $M \ovt M = (LH)^I \rtimes (G \times G)$, associated with $G \times G$ acting on the disjoint union $I := \Gamma \sqcup \Gamma$ of two copies of $\Gamma$. The corresponding tensor length deformation precisely is $\al_t \ot \al_t \in \Aut(\Mtil \ovt \Mtil)$.

Denote by $\delta : \Gamma \recht \Gamma \times \Gamma : \delta(h) = (h,h)$ the diagonal embedding. Observe that the stabilizer (in $G \times G$) of an element $i \in I$ is either of the form $G \times g \delta(\Gamma)g^{-1}$ or $g \delta(\Gamma) g^{-1} \times G$, with $g \in G$. Since $G$ is an icc group, the lemma will follow by applying Theorem \ref{thm.actual-conjugacy} to the generalized Bernoulli action $G \times G \actson (LH)^I$, provided that we prove the following two statements.
\begin{enumerate}
\item $\dis\sup_{g \in G} \|(\al_t \ot \al_t)\Delta(u_g) - \Delta(u_g)\|_2 \recht 0$ as $t \recht 0$.
\item $\Delta(L G) \not\prec M \ovt (A \rtimes \delta(\Gamma))$ and $\Delta(L G) \not\prec (A \rtimes \delta(\Gamma)) \ovt M$.
\end{enumerate}

{\bf Proof of 1.}\ By symmetry, it suffices to prove that
$$\sup_{g \in G} \|(\id \ot \al_t)\Delta(u_g) - \Delta(u_g)\|_2 \recht 0 \quad\text{as}\quad t \recht 0 \; .$$
Since every $g \in G$ is the product of an element in $\Gamma \times \{e\}$ and an element in $\{e\} \times \Gamma$, again by symmetry, it suffices to prove that
\begin{equation}\label{eq.aim-here}
\sup_{g \in \{e\} \times \Gamma} \|(\id \ot \al_t)\Delta(u_g) - \Delta(u_g)\|_2 \recht 0 \quad\text{as}\quad t \recht 0 \; .
\end{equation}
Denote $P := \Delta(L(\Gamma \times \{e\}))$. By Proposition \ref{prop.comult}.\ref{comult-four}, we have that $Pp$ is nonamenable relative to $M^r \ot 1$ for all nonzero projections $p \in P' \cap (M \ovt M)^r$. The unitaries $\Delta(u_g)$, $g \in \{e\} \times \Gamma$, all commute with $P$ and the normalizer of $P$ contains $\Delta(L G)$. So \eqref{eq.aim-here} follows from Lemma \ref{lem.tussenstap}.

{\bf Proof of 2.}\ Since $\delta(\Gamma)$ has infinite index in $G$, statement 2 follows immediately from Lemma \ref{lem.no-embed-Delta-LG}.
\end{proof}

\begin{lemma}\label{lem.embed-rel-comm}
Write $C := \Delta(A_0)' \cap (M \ovt M)^r$. Then $C \prec^f A \ovt A$.
\end{lemma}
\begin{proof}
We start by proving the existence of a nonzero projection $p \in C' \cap (M \ovt M)^r$ such that $C p$ is amenable relative to $M^r \ot 1$. Assume the contrary. Since the normalizer of $C$ contains $\Delta(M_0)$ and since all unitaries in $\Delta(A_0)$ commute with $C$, it follows from Lemma \ref{lem.tussenstap} that
$$\sup_{a \in \cU(A_0)} \|(\id \ot \al_t)\Delta(a) - \Delta(a)\|_2 \recht 0 \quad\text{as $t \recht 0$.}$$
Lemma \ref{lem.conjugate-G} implies in particular that
$$\sup_{g \in G} \|(\id \ot \al_t)\Delta(u_g) - \Delta(u_g)\|_2 \recht 0 \quad\text{as $t \recht 0$.}$$
Note that $\cW := \{a u_g \mid a \in \cU(A_0), g \in G\}$ is a group of unitaries generating $M_0$. The two formulae above imply that
\begin{equation}\label{eq.everywhere}
\sup_{b \in \cW} \|(\id \ot \al_t)\Delta(b) - \Delta(b)\|_2 \recht 0 \quad\text{as $t \recht 0$.}
\end{equation}
We now apply Theorem \ref{thm.actual-conjugacy}. Denote as above $\delta : \Gamma \recht \Gamma \times \Gamma : \delta(h) = (h,h)$. The stabilizer of an element $g \in \Gamma$ under the left-right action $G \actson \Gamma$ can be conjugated into $\delta(\Gamma)$. From Proposition \ref{prop.comult}.\ref{comult-two}, we know that $\Delta(M_0) \not\prec M^r \ovt (A \rtimes \delta(\Gamma))$. So \eqref{eq.everywhere} and Theorem \ref{thm.actual-conjugacy} imply that $\Delta(M_0)$ can be unitarily conjugated into $M^r \ovt L G$. This is in contradiction with Proposition \ref{prop.comult}.\ref{comult-two}.

So we indeed find a nonzero projection $p \in C' \cap (M \ovt M)^r$ such that $C p$ is amenable relative to $M^r \ot 1$. The normalizer of $C$ contains $\Delta(M_0)$ and by Lemma \ref{lem.triv-rel-comm}, we know that $$\Delta(M_0)' \cap (M \ovt M)^r = \C 1 \; .$$ So by Lemma \ref{lem.rel-amen-max-proj}, we conclude that $C$ is amenable relative to $M^r \ot 1$. Applying twice Theorem \ref{thm.norm-amen}, which is possible thanks to Lemma \ref{lem.assum-ok}, it follows that $C \prec^f M^r \ovt (A \rtimes (\Gamma \times \{e\}))$ and that $C \prec^f M^r \ovt (A \rtimes (\{e\} \times \Gamma))$. It then follows from Lemma \ref{lem.full-embed}.\ref{full-two} that $C \prec^f M^r \ovt A$. By symmetry, we also have that $C \prec^f A \ovt M^r$. Again using Lemma \ref{lem.full-embed}.\ref{full-two}, we reach the desired conclusion that $C \prec^f A \ovt A$.
\end{proof}

\begin{lemma}\label{lem.weak-mixing}
If $\cH \subset \Delta(A_0)' \cap (M \ovt M)^r$ is a finite-dimensional, globally $(\Ad \Delta(u_g))_{g \in G}$-invariant subspace, then $\cH \subset \C 1$.
\end{lemma}
\begin{proof}
Put $\cH' := \{x-E_{\Delta(M_0)}(x) \mid x \in \cH\}$. The main part of the proof consists in showing that $\cH' = \{0\}$. Assume on the contrary that $\cH' \neq \{0\}$. Note that $\cH'$ is a finite-dimensional, globally $(\Ad \Delta(u_g))_{g \in G}$-invariant subspace of $\Delta(A_0)' \cap (M \ovt M)^r$ and that $\cH' \subset (M \ovt M)^r \ominus \Delta(M_0)$.

Denote by $\cK$ the closed linear span of $\Delta(M_0) \cH'$ inside $L^2((M \ovt M)^r)$. Observe that $\cK$ is a $\Delta(M_0)$-$\Delta(M_0)$-subbimodule of $L^2((M \ovt M)^r) \ominus L^2(\Delta(M_0))$ that has finite left dimension. By Lemma \ref{lem.discrete-subgroup-Out}, the assumption that $H/H_0$ is torsion-free, and Lemma \ref{lem.comult-rel-commutant-bimod}, there exist automorphisms $\beta_1,\ldots,\beta_k \in \Aut(M_0)$ and a unitary $\psi : \cK \recht L^2(M_0)^{\oplus k} : \xi \mapsto (\psi_1(\xi),\ldots,\psi_k(\xi))$ such that
$$\psi_i(\Delta(x) \, \xi \, \Delta(y)) =  x \, \psi_i(\xi) \, \beta_i(y) \quad\text{for all}\;\; x,y \in M_0, \xi \in \cK, i = 1,\ldots,k,$$
and such that every $\beta_i$ generates a discrete infinite subgroup of $\Out(M_0)$.

Fix an $i \in \{1,\ldots,k\}$ and note that $\psi_i(\cH') \neq \{0\}$. Take a nonzero vector $\xi_0 \in \psi_i(\cH')$. Since the elements of $\cH'$ commute with $\Delta(A_0)$, it follows that that $\xi_0 \beta_i(a) = a \xi_0$ for all $a \in A_0$. By Lemma \ref{lem.an-aut-lem}, we then find a unitary $v \in \cN_M(M_0)$ and a character $\om : G \recht \T$ such that $\beta_i = (\Ad v) \circ \al_\om$ and such that $E_{M_0}(v) = 0$. Recall that $\al_\om(a u_g) = \om(g) \, a u_g$ for all $a \in A_0$ and all $g \in G$.

Put $\cH'_i := \psi_i(\cH') v$ and note that $\cH'_i$ is a finite-dimensional subspace of $L^2(M)$ such that $\xi a = a \xi$ for all $\xi \in \cH'_i$, $a \in A_0$ and such that $\cH'_i$ is globally invariant under $\xi \mapsto \overline{\om(g)} \, u_g \xi u_g^*$ for all $g \in G$. By Proposition \ref{prop.left-right-wreath}.\ref{left-right-c}, we have $A_0' \cap M = A$. So $\cH'_i \subset L^2(A)$. It follows that $\cH'_i$ is a finite-dimensional subspace of $L^2(A)$ that is globally invariant under the generalized Bernoulli action $G \actson A$. By Lemma \ref{lem.prel-weak-mixing}, the latter is weakly mixing. It follows that $\cH'_i \subset \C 1$. So, $\psi_i(\cH') \subset \C v^*$. Since $\psi_i(\cH') \subset L^2(M_0)$, while $v^*$ is orthogonal to $L^2(M_0)$, we find that $\psi_i(\cH') = \{0\}$, which is absurd.

So we have proven that $\cH' = \{0\}$, meaning that $\cH \subset \Delta(M_0)$. So $\cH = \Delta(\cH_0)$ where $\cH_0 \subset A_0' \cap M_0$ is a finite-dimensional, globally $(\Ad u_g)_{g \in G}$-invariant subspace. Since $A_0' \cap M_0 = A_0$ and since the action $G \actson A_0$ is weakly mixing, it follows that $\cH_0 \subset \C 1$. Then also $\cH \subset \C 1$.
\end{proof}

\begin{lemma}\label{lem.strong-intertwine}
We have that $r = 1$ and that there exist a unitary $v \in M_0$, a character $\om : G \recht \T$ and an injective group homomorphism $\rho : G \recht \Lambda$ such that
$$\om(g) \, v u_g v^* = v_{\rho(g)} \;\;\text{for all}\;\; g \in G \quad\text{and}\quad \Delta(v A_0 v^*) \subset v A_0 v^* \ovt v A_0 v^* \; .$$
\end{lemma}
\begin{proof}
We view $M \ovt M$ as the crossed product $M \ovt M = (A \ovt A) \rtimes (G \times G)$. By Proposition \ref{prop.left-right-wreath}.\ref{left-right-c}, we have that $(A \ovt A)' \cap (M \ovt M) = A \ovt A$, meaning that the generalized Bernoulli action $G \times G \actson A \ovt A$ is essentially free. By Lemma \ref{lem.conjugate-G} and after a unitary conjugacy of $\Delta$, we have $\Delta(L G) \subset (L G \ovt L G)^r$. Put $C := \Delta(A_0)' \cap (M \ovt M)^r$.

From Lemma \ref{lem.embed-rel-comm}, we know that $C \prec^f A \ovt A$. By construction, the unitaries $\Delta(u_g)$, $g \in G$, normalize $C$. By Lemma \ref{lem.weak-mixing}, the action $(\Ad \Delta(u_g))_{g \in G}$ on the center $\cZ(C)$ of $C$ is weakly mixing. Actually, Lemma \ref{lem.weak-mixing} says that even the action $(\Ad \Delta(u_g))_{g \in G}$ on $C$ has no nontrivial finite-dimensional invariant subspaces. This means that all the assumptions of \cite[Theorem 6.1]{IPV10} are satisfied. Denote by $N$ the von Neumann algebra generated by $C$ and the unitaries $(\Delta(u_g))_{g \in G}$. Then $\Delta(M_0) \subset N$ and it follows from Proposition \ref{prop.comult}.\ref{comult-two} that $N \not\prec M \ovt (A \rtimes G_1)$ and $N \not\prec (A \rtimes G_1) \ovt M$ whenever $G_1 < G$ has infinite index. So also all the assumptions of \cite[Corollary 6.2]{IPV10} are satisfied. From \cite[Theorem 6.1 and Corollary 6.2]{IPV10}, it then follows that $r = 1$ and that there exist a unitary $\Om_1 \in M \ovt M$, a character $\om : G \recht \T$ and group homomorphisms $\gamma_1,\gamma_2 : G \recht G$ such that
\begin{equation}\label{eq.big-step}
\Om_1 \Delta(A_0) \Om_1^* \subset A \ovt A \quad\text{and}\quad \Om_1 \Delta(u_g) \Om_1^* = \om(g) \, u_{\gamma_1(g)} \ot u_{\gamma_2(g)} \; .
\end{equation}
Since $r = 1$, we may from now on assume that $M_0 = L \Lambda$ and that $\Delta : M_0 \recht M_0 \ovt M_0$ is the original comultiplication given by $\Delta(v_s) = v_s \ot v_s$ for all $s \in \Lambda$.

By \eqref{eq.big-step} and Lemma \ref{lem.no-embed-Delta-LG}, the ranges of $\gamma_1$ and $\gamma_2$ are finite index subgroups of $G$.
Denote by $\zeta : M \ovt M \recht M \ovt M : \zeta(x \ot y) = y \ot x$ the flip automorphism. Since $\zeta \circ \Delta = \Delta$, it follows from \eqref{eq.big-step} that
$$(u_{\gamma_2(g)} \ot u_{\gamma_1(g)}) \; \zeta(\Om_1) \, \Om_1^* \; (u_{\gamma_1(g)}^* \ot u_{\gamma_2(g)}^*) = \zeta(\Om_1) \, \Om_1^*  \quad\text{for all}\;\; g \in G \; .$$
Because $G$ is icc and because the subgroups $\gamma_1(G) < G$ and $\gamma_2(G) < G$ have finite index, we get that $\{(\gamma_2(g) x \gamma_2(g)^{-1},\gamma_1(g) y \gamma_1(g)^{-1}) \mid g \in G\}$ is an infinite set for all $(x,y) \in (\cG \times \cG) - \{e\}$. By Lemma \ref{lem.rel-icc-intertwine}, we then find an $h \in G$ such that $\gamma_1(g) = h \gamma_2(g) h^{-1}$ for all $g \in G$. This means that after replacing $\Om_1$ by $(1 \ot u_h) \Om_1$, we may assume that $\gamma_1 = \gamma_2$. We denote this homomorphism as $\gamma$. It then also follows that $\zeta(\Om_1)$ is a multiple of $\Om_1$. Since $\Delta(u_g)$ and $u_{\gamma(g)} \ot u_{\gamma(g)}$ are unitarily conjugate, the homomorphism $\gamma$ is injective.

Define $K_0 := \widehat{H / H_0}$ and identify $K_0$ with the group of characters on $H$ that are equal to $1$ on $H_0$. Whenever $\eta \in K_0$, the formula
$$\etatil : x g \mapsto \eta\Bigl(\sum_{h \in \Gamma} x_h\Bigr) \quad\text{for all}\;\; x \in H^{(\Gamma)}, g \in G \; ,$$
defines a character on $\cG$ and hence an automorphism $\al_\eta \in \Aut(M)$ by the formula $\al_\eta(u_z) = \etatil(z) u_z$ for all $z \in \cG$. Since $\eta$ equals $1$ on $H_0$, we get that $\al_\eta(a) = a$ for all $a \in M_0$. More precisely, $(\al_\eta)_{\eta \in K_0}$ is a continuous action of $K_0$ on $M$ and the fixed point algebra of this action equals $M_0$.

Let $\eta,\eta' \in K_0$. Applying $\al_\eta \ot \al_{\eta'}$ to \eqref{eq.big-step}, it follows that $\Om_1^* (\al_\eta \ot \al_{\eta'})(\Om_1)$ commutes with $u_{\gamma(g)} \ot u_{\gamma(g)}$ for all $g \in G$. Since $G$ is icc and $\gamma(G) < G$ has finite index, we have that $\{(\gamma(g)x\gamma(g)^{-1},\gamma(g)y\gamma(g)^{-1}) \mid g \in G\}$ is an infinite set for all $(x,y) \in (\cG \times \cG) - \{e\}$. Using Lemma \ref{lem.prel-weak-mixing}, it follows that $\Om_1^* (\al_\eta \ot \al_{\eta'})(\Om_1)$ must be a multiple of $1$ and
we find $\Psi(\eta,\eta') \in \T$ such that
\begin{equation}\label{eq.invariance}
(\al_\eta \ot \al_{\eta'})(\Om_1) = \Psi(\eta,\eta') \, \Om_1 \quad\text{for all}\;\; (\eta,\eta') \in K_0 \times K_0 \; .
\end{equation}
It follows that $\Psi$ is a continuous character on $K_0 \times K_0$. Since $\zeta(\Om_1)$ is a multiple of $\Om_1$, we also get that $\Psi(\eta,\eta') = \Psi(\eta',\eta)$ for all $(\eta,\eta') \in K_0 \times K_0$. Since $\widehat{K_0} = H/H_0$, we find an $x \in H$ such that $\Psi(\eta,\eta') = \eta(x) \eta'(x)$ for all $(\eta,\eta') \in K_0 \times K_0$.

For every $g \in \Gamma$, denote by $\pi_g : L H \recht (L H)^\Gamma$ the embedding of $LH$ as the $g$-th tensor factor. Write $V_x := \pi_e(u_x)$ and put $\Om_2 := (V_x^* \ot V_x^*)\Om_1$. From \eqref{eq.invariance}, it follows that $\Om_2 \in M_0 \ovt M_0$. Denote by $x_e \in H^{(\Gamma)}$ the element $x \in H$ viewed in position $e$. Define the injective group homomorphism $\gamma' : G \recht \cG_0 : \gamma'(g) = x^{-1}_e \gamma(g) x_e$. It follows from \eqref{eq.big-step} that
\begin{equation}\label{eq.crucial-tussen}
\Om_2^* (u_{\gamma'(g)} \ot u_{\gamma'(g)}) \Om_2 = \Delta(\overline{\om(g)} u_g) \quad\text{for all}\;\; g \in G \; .
\end{equation}
Since $\gamma(G)$ has finite index in $G$ and since $G$ is icc, we have that $\{\gamma'(g) x \gamma'(g)^{-1} \mid g \in G \}$ is an infinite set for all $x \in \cG_0 - \{e\}$. By Lemma \ref{lem.prel-weak-mixing}, we get that the representation $(\Ad(u_{\gamma'(g)}))_{g \in G}$ on $L^2(M_0) \ominus \C 1$ is weakly mixing. It then follows from \eqref{eq.crucial-tussen} and \cite[Lemma 3.4]{IPV10} that there exist unitaries $w,v \in M_0$, a character $\om' : G \recht \T$ and an injective group homomorphism $\rho : G \recht \Lambda$ such that
$$w u_{\gamma'(g)} w^* = \om'(g) \, v_{\rho(g)} \;\;\text{for all}\;\; g \in G \quad\text{and}\quad \Om_2 = (w^* \ot w^*) \Delta(v) \; .$$
In combination with \eqref{eq.crucial-tussen}, we get that
$$\om'(g)^2 \, \Delta(v_{\rho(g)}) = \om'(g) \, v_{\rho(g)} \; \ot \; \om'(g) \, v_{\rho(g)} = w u_{\gamma'(g)} w^* \ot w u_{\gamma'(g)} w^*
= \Delta(\overline{\om(g)} \, v u_g v^*)$$
for all $g \in G$. So also $v u_g v^* = \om(g) \om'(g)^2 \, v_{\rho(g)}$. This implies that
$$u_{\gamma'(g)}^* \, w^* v \, u_g = \om(g) \om'(g) \, w^* v \quad\text{for all}\;\; g \in G \; .$$
Lemma \ref{lem.rel-icc-intertwine} then provides an element $k \in \cG_0$ such that $\gamma'(g) = k g k^{-1}$ for all $g \in G$. It follows that $u_k^* w^* v \in \C 1$ and that $\om' = \overline{\om}$. So, $w$ is a multiple of $v u_k^*$ and
$$\om(g) \, v u_g v^* = v_{\rho(g)} \quad\text{for all}\;\; g \in G \; .$$
From \eqref{eq.big-step}, we know that $\Om_2 \Delta(A_0) \Om_2^* \subset A_0 \ovt A_0$. Since $\Om_2 = (w^* \ot w^*) \Delta(v)$ and since $w$ is a multiple of $v u_k^*$, we conclude that $\Delta(v A_0 v^*) \subset v A_0 v^* \ovt v A_0 v^*$.
\end{proof}

\begin{proof}[{\bf Proof of Theorem \ref{thm.main}}]
The proof consists of three different parts.

\subsection*{\boldmath Writing $\Lambda$ as a semidirect product $\Sigma \rtimes G$}

We do not explicitly write the isomorphism $\pi : L \Lambda \recht (L \cG_0)^r$, but directly identify $L \Lambda = L (\cG_0)^r$. We denote by $\Delta : L \Lambda \recht L \Lambda \ovt L \Lambda$ the comultiplication given by $\Delta(v_s) = v_s \ot v_s$ for all $s \in \Lambda$. Recall from \cite[Lemma 7.1]{IPV10} that a von Neumann subalgebra $P \subset L \Lambda$ satisfies $\Delta(P) \subset P \ovt P$ if and only if $P = L S$ for a subgroup $S < \Lambda$.

As above, we denote $\cH := H^{(\Gamma)}$ and $\cH_0 := p_H^{-1}(H_0)$. We write $G := \Gamma \times \Gamma$ and $A := L \cH = (L H)^\Gamma$, with its subalgebra $A_0 := L \cH_0$. Finally, $M := A \rtimes G = L(\cH \rtimes G)$ and $M_0 := A_0 \rtimes G = L \cG_0$. For every character $\om : \cG_0 \recht \T$, we denote by $\al_\om$ the induced automorphism of $M_0$ given by $\al_\om(u_x) = \om(x) u_x$ for all $x \in \cG_0$.

By Lemma \ref{lem.strong-intertwine}, we get that $r = 1$ and that we can compose the identification $L \Lambda = L \cG_0$ with an inner automorphism of $L \cG_0$ and an automorphism of the form $\al_\om$ for a character $\om : G \recht \T$ such that after these compositions, we have
\begin{equation}\label{eq.important}
\Delta(A_0) \subset A_0 \ovt A_0 \quad\text{and}\quad u_g = v_{\rho(g)} \;\;\text{for all}\;\; g \in G \; ,
\end{equation}
where $\rho : G \recht \Lambda$ is an injective group homomorphism. It follows that $A_0 = L \Sigma$ for an abelian subgroup $\Sigma < \Lambda$ and that we have written $\Lambda$ as a semidirect product $\Lambda = \Sigma \rtimes G$, where $G$ acts on $\Sigma$ by group automorphisms. So from now on, we may assume that $\Lambda = \Sigma \rtimes G$ in such a way that $L \Sigma = A_0$ and $v_g = u_g$ for all $g \in G$ (denoting as above by $(v_s)_{s \in \Lambda}$ the canonical unitaries for $L \Lambda$, and by $(u_a)_{a \in \cG_0}$ the canonical unitaries for $L \cG_0$).

\subsection*{\boldmath Proving that $\Sigma$ is of the form $p_{H'}^{-1}(H'_0)$}

Whenever we view $\Gamma$ as the index set of the infinite tensor product $A = (L H)^\Gamma$, we denote the elements of $\Gamma$ by the letters $i,j$. We denote by $g \cdot i$ the left-right action of $g \in G$ on $i \in \Gamma$. We denote by $\pi_i : L H \recht (L H)^\Gamma$ the embedding of $L H$ into $(L H)^\Gamma$ as the $i$-th tensor factor. We denote by $(\si_g)_{g \in G}$ the generalized Bernoulli action given by $\si_g \circ \pi_i = \pi_{g \cdot i}$. We finally denote by $\delta : \Gamma \recht G : \delta(g) = (g,g)$ the diagonal embedding. Since $\Gamma$ is icc, we have that $\delta(\Gamma) \cdot i$ is infinite for all $i \in \Gamma - \{e\}$. By Lemma \ref{lem.prel-weak-mixing}, the action $(\si_{\delta(g)})_{g \in \Gamma}$ on $(L H)^{\Gamma - \{e\}}$ is weakly mixing and we have that
\begin{align}
& \pi_e(L H_0) = \{a \in A_0 \mid \si_{\delta(g)}(a) = a \;\;\text{for all}\;\; g \in \Gamma\} \; , \label{111}\\
& \pi_e(L H_0) \ovt \pi_e(L H_0) = \{a \in A_0 \ovt A_0 \mid (\si_{\delta(g)} \ot \si_{\delta(g)})(a) = a \;\;\text{for all}\;\; g \in \Gamma\} \; , \label{222}\\
& \pi_e(L H) \ovt \pi_e(L H) = \{a \in A \ovt A \mid (\si_{\delta(g)} \ot \si_{\delta(g)})(a) = a \;\;\text{for all}\;\; g \in \Gamma\} \; .\label{333}
\end{align}

For the rest of the proof, we only consider the comultiplication $\Delta$ restricted to $L \Sigma$. Since $v_g = u_g$ for all $g \in G$, we have that $\Delta \circ \si_g = (\si_g \ot \si_g) \circ \Delta$ for all $g \in G$. Using \eqref{111} and \eqref{222}, it then follows that $\Delta(\pi_e(L H_0)) \subset \pi_e(L H_0) \ovt \pi_e(L H_0)$. This means that we find an abelian group $H'_1$ with corresponding comultiplication $\Delta_1 : L H'_1 \recht L H'_1 \ovt L H'_1$, and an identification $L H'_1 = L H_0$ such that $\Delta \circ \pi_e = (\pi_e \ot \pi_e) \circ \Delta_1$.
Composing with $(\si_g \ot \si_g)_{g \in G}$, it follows that $\Delta \circ \pi_i = (\pi_i \ot \pi_i) \circ \Delta_1$ for all $i \in \Gamma$.
So we can view $\pi_i$ as well as an injective group homomorphism of $H'_1$ into $\Sigma$. Since the von Neumann algebras $\pi_i(L H_0)$, $i \in \Gamma$, are in tensor product position inside $L \Sigma$, it follows that the subgroups $\pi_i(H'_1) < \Sigma$, $i \in \Gamma$, are in direct sum position inside $\Sigma$.

Fix an element $x \in H$. The formula $\Om_x(g) := \pi_e(u_x) \pi_{g \cdot e}(u_x^*)$ defines a $1$-cocycle for the action $(\si_g)_{g \in G}$ on $A_0$. Hence $g \mapsto \Delta(\Om_x(g))$ is a $1$-cocycle for the generalized Bernoulli action $(\si_g \ot \si_g)_{g \in G}$ on $(L H)^\Gamma \ovt (L H)^\Gamma$. By Popa's cocycle superrigidity theorem \cite[Theorem 1.1]{Po06b}, we find a unitary $\cV_x \in (L H)^\Gamma \ovt (L H)^\Gamma$ such that
$$\Delta(\Om_x(g)) = \cV_x \, (\si_g \ot \si_g)(\cV_x^*) \quad\text{for all}\;\; g \in G \; .$$
By construction, $\Om_x(\delta(g)) = 1$ for all $g \in \Gamma$. From \eqref{333}, it then follows that $\cV_x = (\pi_e \ot \pi_e)(\cU_x)$ for a unitary $\cU_x \in L H \ovt L H$. So we get that
$$\Delta(\pi_e(u_x) \pi_{g \cdot e}(u_x^*)) = (\pi_e \ot \pi_e)(\cU_x) \, (\pi_{g \cdot e} \ot \pi_{g \cdot e})(\cU_x^*) \quad\text{for all}\;\; x \in H, g \in G \; .$$
Applying $\si_h \ot \si_h$ for an arbitrary $h \in G$, and combining with the earlier definition of $\Delta_1$, we find that
\begin{equation}\label{eq.formulae-comult-Delta}
\begin{split}
& \Delta((\pi_i \ot \pi_j)(u_x \ot u_x^*)) = (\pi_i \ot \pi_i)(\cU_x) \, (\pi_j \ot \pi_j)(\cU_x^*) \quad\text{for all}\;\; x \in H, i,j \in \Gamma \; ,\\
& \Delta \circ \pi_i = (\pi_i \ot \pi_i) \circ \Delta_1 \quad\text{for all}\;\; i \in \Gamma \; .
\end{split}
\end{equation}
Define $H_2 := \{(x,y) \in H \times H \mid x+y \in H_0\}$. Then $H_2$ is generated by the subgroups $H_0 \times H_0$ and $\{(x,-x) \mid x \in H\}$.
Since $L H'_1 = L H_0$, the von Neumann algebra generated by the elements $\{(\pi_i \ot \pi_j)(u_x \ot u_x^*) \mid i,j \in \Gamma, x \in H\}$, together with the algebras $\pi_i(L H'_1)$, $i \in \Gamma$, equals the von Neumann algebra generated by all the $(\pi_i \ot \pi_j)(L H_2)$, which is the whole of $A_0= L \Sigma$.

So the formulae in \eqref{eq.formulae-comult-Delta} entirely determine $\Delta$.
Also note that for a given $x \in H$, the unitary $\cU_x$ is uniquely determined up to multiplication by a scalar in $\T$. Finally observe that for $x \in H_0$, we have $\cU_x = \Delta_1(u_x)$, up to multiplication by a scalar in $\T$. In particular, $\cU_x \in L H_0 \ovt L H_0$ whenever $x \in H_0$.

For all distinct $i,j \in \Gamma$, denote by $\pi_{ij} : L H_2 \recht A_0$ the embedding into the $i$'th and $j$'th coordinate.
It follows from
\eqref{eq.formulae-comult-Delta} that we can identify $L H_2 = L H'_2$ for some abelian group $H'_2$ with the corresponding comultiplication $\Delta_2 : L H'_2 \recht L H'_2 \ovt L H'_2$ given by
the following formulae that use the tensor leg numbering notation.
\begin{equation}\label{eq.formulae-comult-Delta-two}
\begin{split}
& \Delta_2(u_x \ot u_x^*) = (\cU_x)_{13} \, (\cU_x^*)_{24} \quad\text{for all}\;\; x \in H \; ,\\
& \Delta_2(a \ot b) = (\Delta_1(a))_{13} \, (\Delta_1(b))_{24} \quad\text{for all}\;\; a,b \in LH'_1 \; .
\end{split}
\end{equation}
By construction, we have $\Delta \circ \pi_{ij} = (\pi_{ij} \ot \pi_{ij}) \circ \Delta_2$. So we can view $\pi_{ij}$ as an injective group homomorphism $\pi_{ij} : H'_2 \recht \Sigma$. Note that we can naturally view $H'_1 \times H'_1$ as a subgroup of $H'_2$ and that under this identification $\pi_{ij}(a,b) = \pi_i(a) + \pi_j(b)$ for all $(a,b) \in H'_1 \times H'_1$.

We denote by $K := \widehat{H}$ the group of characters on $H$ and by $K_0 < K$ the closed subgroup of characters that are identically $1$ on $H_0$. We identify $K_0 = \widehat{H/H_0}$. Whenever $\om \in K$, we denote by $\al_\om \in \Aut(L H)$ the induced automorphism given $\al_\om(u_x) = \om(x) u_x$ for all $x \in H$. Applying $\al_\om$ in the $i$-th coordinate yields the automorphism $\al^i_\om \in \Aut((LH)^\Gamma)$, while applying $\al_\om$ in all coordinates yields the automorphism $\al^\Gamma_\om \in \Aut((LH)^\Gamma)$. By construction, we have that $\al^\Gamma_\om \circ \pi_i = \pi_i \circ \al_\om$. A given $a \in A= (LH)^\Gamma$ belongs to $A_0$ if and only if $\al^\Gamma_\om(a) = a$ for all $\om \in K_0$.

Fix $x \in H$. Since $\Delta(A_0) \subset A_0 \ovt A_0$, the left hand side of the formulae in \eqref{eq.formulae-comult-Delta} is invariant under $\al^\Gamma_\om \ot \id$ for all $\om \in K_0$. Since $\cU_x$ is uniquely determined up to a scalar, it follows that $(\al_\om \ot \id)(\cU_x)$ is a multiple of $\cU_x$ for every $\om \in K_0$. So we find an element $\gamma(x) \in H/H_0$ such that
$$(\al_\om \ot \id)(\cU_x) = \om(\gamma(x)) \, \cU_x \quad\text{for all}\;\; \om \in K_0 \; .$$
When $x \in H_0$, we have that $\cU_x \in LH_0 \ovt LH_0$ and hence $\gamma(x) = 0$. It follows that $\gamma$ is a well-defined group homomorphism from $H/H_0$ to $H/H_0$.

The formulae in \eqref{eq.formulae-comult-Delta} entirely determine $\Delta$ so that it follows that $(\al^i_\om \ot \id) \circ \Delta = \Delta \circ \al^i_{\om \circ \gamma}$ for all $i \in \Gamma$ and all $\om \in K_0$. Using Lemma \ref{lem.descr-aut}, we conclude that $\gamma = \id$ and that every automorphism $\al^i_\om$ is induced by a character of $\Sigma$. It follows that there are group homomorphisms $\psi_i : \Sigma \recht H/H_0$ such that
$$\al^i_\om(v_s) = \om(\psi_i(s)) \, v_s \quad\text{for all}\;\; s \in \Sigma, i \in \Gamma, \om \in K_0 \; .$$
A similar reasoning, using \eqref{eq.formulae-comult-Delta-two} instead of \eqref{eq.formulae-comult-Delta}, provides a homomorphism $\psi : H'_2 \recht H/H_0$ such that $(\al_\om \ot \id)(v_s) = \om(\psi(s)) \, v_s$ for all $s \in H'_2$, $\om \in K_0$.

Since $\al^i_\om \circ \pi_{ij} = \pi_{ij} \circ (\al_\om \ot \id)$, we have that $\psi_i \circ \pi_{ij} = \psi$. Since $(\al_{\overline{\om}} \ot \id)(x) = (\id \ot \al_\om)(x)$ for all $x \in L H'_2$, we have $\al^j_\om \circ \pi_{ij} = \pi_{ij} \circ (\al_{\overline{\om}} \ot \id)$. Hence $\psi_j \circ \pi_{ij} = -\psi$.
We further have that $\psi_k \circ \pi_{ij} = 0$ if $k \not\in \{i,j\}$.

We already observed above that the subgroups $\pi_i(H'_1) < \Sigma$, $i \in \Gamma$, are in a direct sum position. Denote by $\Sigma_1 < \Sigma$ the subgroup generated by the $\pi_i(H'_1)$, $i \in \Gamma$. Since $L H'_1 = L H_0$, we have that $L \Sigma_1 = (L H_0)^\Gamma$. It follows that
$$L \Sigma_1 = \{x \in A_0 \mid \al^i_\om(x) = x \;\;\text{for all}\;\; i \in \Gamma, \om \in K_0 \} \;\;\text{and hence}\;\; \Sigma_1 = \bigcap_{i \in I} \Ker \psi_i \; .$$

Every permutation $\beta \in \Perm \Gamma$ defines an automorphism $\gamma_\beta$ of $(LH)^\Gamma$ by permuting the tensor factors. It follows from \eqref{eq.formulae-comult-Delta} that $(\gamma_\beta \ot \gamma_\beta) \circ \Delta = \Delta \circ \gamma_\beta$, so that $\gamma_\beta$ induces a group automorphism of $\Sigma$. By construction, we have $\gamma_\beta \circ \pi_i = \pi_{\beta(i)}$ and $\gamma_\beta \circ \pi_{ij} = \pi_{\beta(i),\beta(j)}$.

It is now easy to check that all assumptions of Lemma \ref{lem.abstract-combinatorial} are satisfied. We conclude from Lemma \ref{lem.abstract-combinatorial} that there exists an abelian group $H'$ with subgroup $H'_0 < H'$ and a $G$-equivariant group isomorphism $p_{H'}^{-1}(H'_0) \recht \Sigma$.

\subsection*{\boldmath Proving that the isomorphism $\pi$ is of the required form}

We put $\cH' := {H'}^{(\Gamma)}$ and $\cH'_0 := p_{H'}^{-1}(H_0')$. Precomposing the original identification of $L \Sigma$ and $L \cH_0$, with the above identification of $L \Sigma$ and $L \cH'_0$, we have brought us to the point where $\Lambda = \cH'_0 \rtimes G$ and where the isomorphism
$$\pi : L(\cH'_0 \rtimes G) \recht L(\cH_0 \rtimes G)$$
satisfies $\pi(L \cH'_0) = L \cH_0$ and $\pi(u_g) = u_g$ for all $g \in G$.

Denote by $\vphi : L \cH'_0 \recht L \cH_0$ the restriction of $\pi$ to $L \cH'_0$. Note that $\vphi$ is a $G$-equivariant $*$-isomorphism. To conclude the proof of Theorem \ref{thm.main}, it remains to prove that $\vphi$ must be of the following special form: there exist a group isomorphism $\gamma : H'/H'_0 \recht H/H_0$, a $G$-invariant character $\mu : \cH_0 \recht \T$ and a trace preserving $*$-isomorphism $\vphi_0 : L H' \recht L H$ such that $\vphi_0 \circ \al_{\om \circ \gamma} = \al_\om \circ \vphi_0$ for all $\om \in \widehat{H/H_0}$ and such that $\vphi = \al_\mu \circ \vphi_0^\Gamma$. Here the $*$-isomorphism $\vphi_0^\Gamma : (L H')^\Gamma \recht (L H)^\Gamma$ is defined as the infinite tensor product of copies of $\vphi_0$.

Denote $K = \widehat{H}$, $K' = \widehat{H'}$, $K_0 = \widehat{H/H_0}$ and $K'_0 = \widehat{H'/H'_0}$. Consider the compact group $K^\Gamma$ and embed $K_0$ as a subgroup of $K^\Gamma$ diagonally. We similarly consider $K'_0 < {(K')}^\Gamma$. We identify
$$L \cH'_0 = L^\infty\Bigl(\frac{(K')^\Gamma}{K'_0}\Bigr) \quad\text{and}\quad L \cH_0 = L^\infty\Bigl(\frac{K^\Gamma}{K_0}\Bigr) \; .$$
We can then view $\vphi = \theta_*$ where $\theta$ is a probability measure preserving (pmp), $G$-equivariant isomorphism
$$\theta :  \frac{(K')^\Gamma}{K'_0} \recht \frac{K^\Gamma}{K_0} \; .$$
Consider the natural actions $G \times K'_0 \actson {(K')}^\Gamma$ and $G \times K_0 \actson K^\Gamma$. By Popa's cocycle superrigidity theorem \cite[Theorem 1.1]{Po06b} and \cite[Lemma 5.2]{PV06}, there exist a pmp isomorphism\linebreak $\thetatil : {(K')}^\Gamma \recht K^\Gamma$, a group homomorphism $\beta : G \recht K_0 : g \mapsto \beta_g$ and a continuous group isomorphism $\gammah : K'_0 \recht K'$ such that
\begin{equation}\label{eq.formula-thetatil}
\thetatil((g,k) \cdot \om) = (g,\beta_g \gammah(k)) \cdot \thetatil(\om) \quad\text{and}\quad \thetatil(\om) + K_0 = \theta(\om + K'_0) \; ,
\end{equation}
for all $(g,k) \in G \times K'_0$ and a.e.\ $\om \in {(K')}^\Gamma$.

Fix $x \in H$ and denote $F_x : K^\Gamma \recht \T : F_x(\om) = \om_e(x)$. As before, denote by $\delta : \Gamma \recht G : \delta(g) = (g,g)$ the diagonal embedding. One checks that
$$(F_x \circ \thetatil)(\delta(g) \cdot \om) = \beta_g(x) \, (F_x \circ \thetatil)(\om) \;\;\text{for all}\;\; g \in \Gamma \;\;\text{and a.e.}\;\; \om \in {(K')}^\Gamma \; .$$
Since $\Gamma$ is icc, it follows from Lemma \ref{lem.prel-weak-mixing} that the action of $\delta(\Gamma)$ on ${(K')}^{\Gamma-\{e\}}$ is weakly mixing, so that the function $\om \mapsto (F_x \circ \thetatil)(\om)$ only depends on the coordinate $\om_e$. Since this holds for all $x \in H$, we find a pmp isomorphism $\theta_0 : K' \recht K$ such that $(\thetatil(\om))_e = \theta_0(\om_e)$ for a.e.\ $\om$. By construction, we have $\theta_0(k + \om) = \gammah(k) + \theta_0(\om)$ for all $k \in K'_0$ and a.e.\ $\om \in K'$. Writing $\vphi_0 := (\theta_0)_*$, we obtain the trace preserving $*$-isomorphism $\vphi_0 : L H' \recht L H$ satisfying $\vphi_0 \circ \al_{\om \circ \gamma} = \al_\om \circ \vphi_0$ for all $\om \in \widehat{H/H_0}$.

Evaluating \eqref{eq.formula-thetatil} in the coordinate $e$, we find that $\beta_{\delta(g)} = 0$ for all $g \in \Gamma$, so that $\beta_{(g,h)} = \rho_g - \rho_h$ for a group homomorphism $\rho : \Gamma \recht K_0 : g \mapsto \rho_g$. We also find that $\thetatil(\om)_g = \theta_0(\om_g) + \rho_g$ for all $g \in \Gamma$ and a.e.\ $\om \in {(K')}^\Gamma$. Define $\mu \in K^\Gamma / K_0$ as $\mu := (\rho_g)_{g \in \Gamma} + K_0$. Then $\mu$ is a $G$-invariant element of $K^\Gamma / K_0$, i.e.\ a $G$-invariant character on $\cH_0$. By construction, we have that $\vphi = \al_\mu \circ \vphi_0^\Gamma$.
\end{proof}

\subsection*{A combinatorial lemma}

Whenever $I$ is a countable set and $H$ is a countable abelian group with subgroup $H_0 < H$, we consider the direct sum $H^{(I)}$, the group homomorphism
$$p_H : H^{(I)} \recht H : p_H(x) = \sum_{g \in I} x_g$$
and the subgroup $p_H^{-1}(H_0)$ of $H^{(I)}$. The group $\Perm I$ of all permutations of $I$ acts on $H^{(I)}$ by group automorphisms that leave the subgroup $p_H^{-1}(H_0)$ globally invariant.

For every $i \in I$, we have a natural embedding $\mu_i : H_0 \recht p_H^{-1}(H_0)$ of $H_0$ into the $i$-th coordinate. Writing $(H \times H)_{H_0} := \{(x,y) \in H \times H \mid x+y \in H_0\}$, we also have natural embeddings $\mu_{ij} : (H \times H)_{H_0} \recht p_H^{-1}(H_0)$ into the $i$-th and $j$-th coordinate, whenever $i$ and $j$ are distinct elements of $I$. The subgroups $\mu_{ij}((H \times H)_{H_0})$ generate $p_H^{-1}(H_0)$.

The following elementary lemma abstractly characterizes this whole setup. The lemma is actually much more awkward to state than to prove.

\begin{lemma}\label{lem.abstract-combinatorial}
Let $\Sigma$ be a countable abelian group and $I$ a countably infinite set. Assume that we are given the following data:
\begin{itemize}
\item countable abelian groups $H_1$ and $H_2$ such that $H_1 \times H_1 < H_2$,
\item for all $i \in I$, an injective homomorphism $\pi_i : H_1 \recht \Sigma$,
\item for all distinct $i,j \in I$, an injective homomorphism $\pi_{ij} : H_2 \recht \Sigma$,
\item an abelian group $L$ and, for all $i \in I$, a group homomorphism $\psi_i : \Sigma \recht L$,
\item a group homomorphism $\psi : H_2 \recht L$,
\item an action of the group of all permutations $\beta \in \Perm I$ by group automorphisms $\gamma_\beta$ of $\Sigma$,
\end{itemize}
such that the following conditions hold:
\begin{itemize}
\item the subgroups $\pi_{ij}(H_2)$ generate $\Sigma$,
\item the subgroups $\pi_i(H_1)$ are in a direct sum position inside $\Sigma$ and generate a subgroup of $\Sigma$ denoted by $\Sigma_1$,
\item we have $\pi_{ij}(a,b) = \pi_i(a) + \pi_j(b)$ for all $(a,b) \in H_1 \times H_1 \subset H_2$,
\item we have $\psi_i \circ \pi_{ij} = \psi = - \psi_j \circ \pi_{ij}$,
\item we have $\psi_k \circ \pi_{ij} = 0$ if $k \not\in \{i,j\}$,
\item we have $\Sigma_1 = \bigcap_{i \in I} \Ker \psi_i$,
\item for every $\beta \in \Perm I$, we have $\gamma_\beta \circ \pi_i = \pi_{\beta(i)}$ and $\gamma_\beta \circ \pi_{ij} = \pi_{\beta(i) , \beta(j)}$.
\end{itemize}
Then there exist a countable abelian group $H$ with subgroup $H_0 < H$ and group isomorphisms
$$\delta_1 : H_0 \recht H_1 \quad , \quad \delta_2 :  (H \times H)_{H_0} \recht H_2 \quad\text{and}\quad \delta : p_H^{-1}(H_0) \recht \Sigma$$
such that, using the notations $\mu_i$ and $\mu_{ij}$ introduced before the lemma, we have
\begin{itemize}
\item $\delta$ conjugates the actions of $\Perm I$,
\item $\delta \circ \mu_i = \pi_i \circ \delta_1$,
\item $\delta \circ \mu_{ij} = \pi_{ij} \circ \delta_2$.
\end{itemize}
\end{lemma}
\begin{proof}
We may assume that $I = \N$. Since the subgroups $\pi_i(H_1) < \Sigma$ are in a direct sum position, we can assemble the $\pi_i$ into an isomorphism $\pi : H_1^{(\N)} \recht \Sigma_1$. Note that $\pi$ conjugates the natural actions of $\Perm \N$.

Fix $x \in H_2$. Observe that $y := \pi_{12}(x) + \pi_{23}(x) + \pi_{31}(x)$ belongs to the kernel of all $\psi_i$, $i \in \N$. Hence, $y = \pi(z)$ for some element $z \in H_1^{(\N)}$. It follows that $z$ is invariant under cyclic permutations of $(1,2,3)$. It also follows that $z$ is invariant under all permutations that fix $1$, $2$ and $3$. Since there are only finitely many $k \in \N$ with $z_k \neq 0$, we conclude that $y$ must be of the form $y = \pi_1(\rho(x)) + \pi_2(\rho(x)) + \pi_3(\rho(x))$, where $\rho : H_2 \recht H_1$ is a group homomorphism. Also note that $\rho(a,b) = a+b$ for all $(a,b) \in H_1 \times H_1 \subset H_2$.

We define $H := \Ker \rho$. We define the subgroup $H_0 < H$ given by $H_0 := \{(a,-a) \mid a \in H_1\}$. We denote $\delta_1 : H_0 \recht H_1 : \delta_1(a,-a) := a$.

By construction, we have that $\pi_{12}(x) + \pi_{23}(x) + \pi_{31}(x) = 0$ for all $x \in H$. Applying $\gamma_\beta$ for an arbitrary permutation $\beta$ of $\N$, it follows that
\begin{equation}\label{eq.starstar}
\pi_{ij}(x) + \pi_{jk}(x) + \pi_{ki}(x) = 0
\end{equation}
for all $x \in H$ and all distinct $i,j,k \in \N$.

Fix $x \in H_2$. Observe that $y := \pi_{12}(x) + \pi_{21}(x)$ belongs to the kernel of all $\psi_i$, $i \in \N$. We also have that $\gamma_\beta(y) = y$ when $\beta$ is the permutation of $\N$ that flips $1$ and $2$, as well as when $\beta$ is a permutation that fixes $1$ and $2$. Reasoning as above, it follows that $\pi_{12}(x) + \pi_{21}(x) = - \pi_1(\eta(x)) - \pi_2(\eta(x))$, where $\eta : H_2 \recht H_1$ is a group homomorphism. We only introduced the minus sign to make the following computation easier. Applying $\gamma_\beta$ for an arbitrary permutation $\beta$ of $\N$, we get that
$$\pi_{ji}(x) = -\pi_{ij}(x) + \pi_i(\eta(x)) + \pi_j(\eta(x))$$
for all $x \in H_2$ and all distinct $i,j \in \N$.

We prove that $\eta(x) = 0$ for all $x \in H$. Fix $x \in H$ and consider the element
$$y := \pi_{12}(x) + \pi_{23}(x) + \pi_{34}(x) + \pi_{41}(x) \; .$$
A first computation, using \eqref{eq.starstar}, yields
\begin{align*}
y &= -\pi_{31}(x) + \pi_{34}(x) + \pi_{41}(x) = \pi_1(\eta(x)) + \pi_3(\eta(x)) + \pi_{13}(x) + \pi_{34}(x) + \pi_{41}(x) \\
&= \pi_1(\eta(x)) + \pi_3(\eta(x)) \; .
\end{align*}
An analogous second computation gives
\begin{align*}
y &= \pi_{12}(x) - \pi_{42}(x) + \pi_{41}(x) = \pi_2(\eta(x)) + \pi_4(\eta(x)) + \pi_{12}(x) + \pi_{24}(x) + \pi_{41}(x) \\
&= \pi_2(\eta(x)) + \pi_4(\eta(x)) \; .
\end{align*}
Since the groups $\pi_i(H_1)$ are in a direct sum position inside $\Sigma$, both computations together imply that $\eta(x) = 0$ for all $x \in H$. It follows that
$\pi_{ij}(x) = - \pi_{ji}(x)$ for all $x \in H$ and all distinct $i,j \in \N$. In combination with \eqref{eq.starstar}, we get that
\begin{equation}\label{eq.crucial}
\pi_{ij}(x) + \pi_{jk}(x) = \pi_{ik}(x)
\end{equation}
for all $x \in H$ and all distinct $i,j,k \in \N$.

We claim that the homomorphism
$$\delta_2 : (H \times H)_{H_0} \recht H_2 : \delta_2(x,y) = x + (0,\delta_1(x+y))$$
is an isomorphism of groups satisfying $\delta_2(x,y) = (\delta_1(x),\delta_1(y))$ for all $(x,y) \in H_0 \times H_0$. This last formula is immediate. It already implies that the image of $\delta_2$ contains both $H$ and $H_1 \times H_1$. Since for every $x \in H_2$, we have that $x - (0,\rho(x)) \in H$, the surjectivity of $\delta_2$ follows. Since $\rho(\delta_2(x,y)) = \delta_1(x+y)$, the injectivity of $\delta_2$ follows as well.

Using \eqref{eq.crucial}, it follows that the formula
$$\delta : p_H^{-1}(H_0) \recht \Sigma : \delta(x) = \pi_{n+1}(\delta_1(p_H(x))) + \sum_{i=1}^n \pi_{i,n+1}(x_i) \quad\text{whenever}\;\; x_k = 0 \;\;\text{for all}\;\; k > n$$
is independent of the choice of $n$ and hence a well-defined homomorphism satisfying $\delta \circ \mu_{ij} = \pi_{ij} \circ \delta_2$ and $\delta \circ \mu_i = \pi_i \circ \delta_1$. It immediately follows that $\delta$ conjugates the respective actions of $\Perm \N$ and that $\delta$ is surjective.

To prove the injectivity of $\delta$, we first claim that $H_0 = H \cap \Ker \psi$. The inclusion $\subset$ is obvious. Conversely, assume that $y \in H$ and $\psi(y) = 0$. Put $z = \pi_{12}(y)$. We get that $z \in \Ker \psi_k$ for all $k \in \N$. So $z \in \Sigma_1$. Since $\gamma_\beta(z) = z$ for every permutation $\beta$ that fixes $1$ and $2$, we find that $y \in H_1 \times H_1$. Since $y \in H$, we obtain the claim that $y \in H_0$. If now $\delta(x) = 0$, we get that $\psi(x_i) = \psi_i(\delta(x)) = 0$ for all $i \in \N$. So $x$ belongs to $H_0^{(I)}$. Since $\delta \circ \mu_i = \pi_i \circ \delta_1$, the restriction of $\delta$ to $H_0^{(I)}$ is injective.
\end{proof}

\subsection*{Proofs of Theorem \ref{thm.main-intro} and Remark \ref{rem.funny-case}}

\begin{proof}[{\bf Proof of Theorem \ref{thm.main-intro}}]
A hyperbolic group $\Gamma$ has only finitely many conjugacy classes of finite subgroups (see e.g.\ \cite{Br99}). By Selberg's lemma \cite{Se60}, a finitely generated linear group $\Gamma$ (over a field of characteristic zero) has a finite index subgroup that is torsion-free. In both cases, $\Gamma$ admits a bound on the possible orders of its finite subgroups. By the work of \cite{CH88,Sk88,Oz03,Oz07} (see \cite[Lemma 2.4]{PV12} for a more detailed explanation), we also have in both cases that $\Gamma$ is weakly amenable and that $\Gamma$ belongs to class $\cS$.
So every group $\Gamma$ that appears in Theorem \ref{thm.main-intro} satisfies the conditions of Theorem \ref{thm.main}.

We will apply Theorem \ref{thm.main}. The conclusion of Theorem \ref{thm.main} describes the given $*$-isomorphism $\pi : L \Lambda \recht (L \cG_0)^r$ as a composition of an inner automorphism, ``group like'' isomorphisms implemented by group isomorphisms and characters, and the $*$-isomorphism $\pi_\theta$ that need not be group like in general. We now prove that in the situation of Theorem \ref{thm.main-intro}, also $\pi_\theta$ is group like.

1.\ Assume that $H = \Z/ n \Z$ with $n \in \{2,3\}$ and put $\cG = H^{(\Gamma)} \rtimes (\Gamma \times \Gamma)$. We apply Theorem \ref{thm.main} with $H_0 = H$. This provides an abelian group $H'$ with $|H'| = |H|$. So, $H' \cong H$ and we may assume that $H' = H$. It only remains to prove that the automorphism $\pi_\theta : L \cG \recht L \cG$ is group like. But since $L H$ has dimension $2$ or $3$, it is not hard to check that every automorphism $\theta : L H \recht L H$ is of the form $\theta = \al_\om \circ \pi_\delta$ for some character $\om \in \widehat{H}$ and group automorphism $\delta : H \recht H$. Then $\pi_\theta$ is group like as well.

2.\ We apply Theorem \ref{thm.main} with $H_0 = \{0\}$. Since $H' \cong H$, we may assume that $H' = H$. Then $\theta : \widehat{H} \recht \widehat{H}$ is a pmp isomorphism satisfying $\theta(k+\om) = k + \theta(\om)$ for a.e.\ $k,\om \in \widehat{H}$. So we find a fixed $\om_0 \in \widehat{H}$ such that $\theta(\om) = \om + \om_0$ for a.e.\ $\om \in \widehat{H}$. But then $\pi_\theta$ is the identity map.
\end{proof}

\begin{proof}[{\bf Proof of Remark \ref{rem.funny-case}}]
Assume that $\Gamma$ has no nontrivial characters. Put $G = \Gamma \times \Gamma$, $\cH_0 = p_H^{-1}(\{0\})$ and $\cG_0 = \cH_0 \rtimes G$. Put $K = \widehat{H}$. Since $G$ has no nontrivial characters, we only need to prove that $\cH_0$ has no nontrivial $G$-invariant characters. This means that we have to prove that the action of $G$ on the compact space $K^\Gamma / K$ only has $0$ as a fixed point. One checks that the $G$-fixed points in $K^\Gamma / K$ are precisely the points $(\al_g)_{g \in \Gamma} + K$ where $\al : \Gamma \recht K$ is a homomorphism. Since $\Gamma$ has no nontrivial characters and $K$ is abelian, such a homomorphism is constantly equal to $0$.
\end{proof}

\end{document}